\documentclass[11pt,twoside,a4paper]{article}

\usepackage[english]{babel}
\usepackage[latin1]{inputenc}
\usepackage[T1]{fontenc} 
\usepackage{dsfont}
\usepackage{amssymb,amsmath,amsthm,amscd,mathrsfs}

\usepackage{tikz} 
\usetikzlibrary{matrix,arrows,backgrounds,calc,chains}
\usetikzlibrary{automata}
\usetikzlibrary{positioning,patterns}
\usetikzlibrary{decorations.pathreplacing}
\usetikzlibrary{decorations.fractals}
\usetikzlibrary{arrows}
\usepackage{graphicx, subfigure} 
\usepackage[arrow, matrix, curve]{xy}

\theoremstyle{plain} \numberwithin{equation}{section}
\newtheorem{thm}{Theorem}[section]
\newtheorem{lemma}[thm]{Lemma}

\newtheorem{cor}[thm]{Corollary}
\newtheorem{prop}[thm]{Proposition}
\theoremstyle{definition}
\newtheorem{defn}[thm]{Definition}
\theoremstyle{remark}
\newtheorem{rem}[thm]{Remark}
\theoremstyle{remark}

\providecommand{\abs}[1]{\left\vert#1\right\vert}
\providecommand{\norm}[1]{\left\Vert#1\right\Vert}

\providecommand{\brac}[1]{\left\{#1\right\}}

\DeclareMathOperator{\diam}{diam}

\DeclareMathOperator{\supp}{supp}

\newcommand{\R}{{\mathbb R}}

\newcommand{\N}{{\mathbb N}}

\newcommand{\Vs}{V_{\ast}}
\newcommand{\Ka}{K_{\alpha}}
\newcommand{\Kao}{K_{\alpha}^0}
\newcommand{\Van}{V_{\alpha,n}}
\newcommand{\Vas}{V_{\alpha,\ast}}
\newcommand{\Fa}{F_{\alpha}}
\newcommand{\Ja}{J_{\alpha}}
\newcommand{\Jae}{\mathcal{J}_{\alpha}}
\newcommand{\Wan}{W_{\alpha,n}}
\newcommand{\Was}{W_{\alpha,\ast}}
\newcommand{\Jan}{J_{\alpha,n}}

\newcommand{\Gai}{G_{\alpha,i}}
\newcommand{\Gaw}{G_{\alpha,w}}
\newcommand{\A}{\mathcal{A}}
\newcommand{\An}{\mathcal{A}^n}
\newcommand{\Am}{\mathcal{A}^m}
\newcommand{\As}{\mathcal{A}^{\ast}}

\newcommand{\ean}{E_{\alpha,n}}
\newcommand{\rhoan}{\rho_{\alpha,n}}

\newcommand{\dai}{d_{\alpha,i}}
\newcommand{\Ean}{\mathcal{E}_{\alpha,n}}

\newcommand{\EKa}{\mathcal{E}_{K_{\alpha}}}
\newcommand{\EKao}{\mathcal{E}_{K_{\alpha}^0}}
\newcommand{\FKa}{\mathcal{F}_{K_{\alpha}}}
\newcommand{\Da}{\mathcal{D}_{K_{\alpha}}}
\newcommand{\Dao}{\mathcal{D}_{K_{\alpha}^0}}
\newcommand{\Dan}{\mathcal{D}_{\alpha,n}}
\newcommand{\mua}{\mu_{\alpha,\beta}}
\newcommand{\muad}{\mu_{\alpha}^d}
\newcommand{\muac}{\mu_{\alpha,\beta}^c}

\newcommand{\LaN}{\Delta^N_{\mu_{\alpha,\beta}}}
\newcommand{\LaD}{\Delta^D_{\mu_{\alpha,\beta}}}

\newcommand{\La}{\Delta_{\mu_{\alpha,\beta}}}

\providecommand{\Kak}[1]{K_{\alpha,{#1}}}
\providecommand{\Ga}[1]{G_{\alpha,{#1}}}

\providecommand{\Jak}[1]{J_{\alpha,{#1}}}
\providecommand{\Jake}[1]{\mathcal{J}_{\alpha,{#1}}}
\providecommand{\Va}[1]{V_{\alpha,{#1}}}

\newcommand{\D}{{\mathcal D}}

\newcommand{\E}{\mathcal{E}}
\newcommand{\F}{\mathcal{F}}

\newcommand{\J}{\mathcal{J}}

\usepackage{verbatim}
\usepackage{hyperref}

\usepackage{setspace}

\begin{document}
\title{Weyl asymptotics for Hanoi attractors}

\author{Alonso Ruiz, P. and Freiberg, U.R.}
\date{}

\maketitle

\begin{abstract}
The asymptotic behavior of the eigenvalue counting function of Laplacians on Hanoi attractors is determined. To this end, Dirichlet and resistance forms are constructed. Due to the non self-similarity of these sets, the classical construction of the Laplacian for p.c.f. self-similar fractals has to be modified by combining discrete and quantum graph methods.
\end{abstract}

\tableofcontents
\section{Introduction}

It is a well known fact from the theory of Dirichlet forms which can be found e.g. in~\cite{FOT11}, that any local and regular Dirichlet form defines a diffusion process on a set. The development of this theory when the underlying set is fractal started with the construction of Brownian motion on the Sierpi\'nski gasket by Goldstein and Kusuoka in~\cite{Gol87,Kus89}. Since then, many results concerning both Dirichlet forms and diffusion processes on fractals have been established. The self-similar case was first discussed in~\cite{BP88,Kig89,Kig93} on p.c.f. sets and later~\cite{Kaj10} discussed results for the Sierpi\'nski carpet. Non strictly self-similarity can be obtained by introducing randomnes, as the case of homogeneous random p.c.f. fractals and carpets treated in~\cite{Ham00,HK01}, or constructing deterministic examples like self-conformal IFS's, treated in~\cite{AHS08,FL05}, fractafolds~\cite{Str03,ST12}, fractal fields~\cite{HK03}, or fractal quantum graphs~\cite{AKT14}.

\medskip

In this paper, we would like to consider diffusion on a special type of non self-similar sets that we call \textit{Hanoi attractors of parameter} $\alpha$, with $\alpha\in (0,1/3)$. Similar objects have been recently investigated from a topological point of view in~\cite{Geo14}; an stochastic approach of the construction of diffusion in that case has appeared in~\cite{GK14}.

\medskip

Hanoi attractors can be considered as (degenerated) graph directed fractals, introduced in~\cite{MW88}, where the contractions associated to some of the edges are not similitudes. An analysis for such objects was first treated in~\cite{Mur95} for the special case of the \textit{plain Mandala}, and it was generalized in~\cite{HN03} for any graph directed fractal. Here the Laplacian is constructed via Dirichlet forms and its spectral asymptotics are calculated. Our work differs from this in that we construct first a resistance form and afterwards choose a measure that allows us to compute the second term of the spectral asympotics of the Laplacian associated to the induced Dirichlet form. The theory of resistance forms provides a more general framework and was introduced by Kigami in~\cite{Kigami01}. It has been broadly studied in the context of self-similar and p.c.f. sets in~\cite{Kig03,Tep08}.

\medskip

Our interest in Hanoi attractors lies in their geometric relationship with the Sierpi\'nski gasket (see Theorem~\ref{thm: Properties of Ka} below and~\cite{AF12} for further details). The main question we would like to answer here is if these objects are also analytically related in terms of spectral dimension.

\medskip

We recall briefly the construction of Hanoi attractors: let us denote by $\mathscr{H}(\R^2)$ the space of nonempty compact subsets of $\R^2$ and equip it with the Hausdorff distance $h$ given by
\[h(A,B):=\inf\brac{\varepsilon>0~\vert~A\subseteq B_{\varepsilon}\text{ and }B\subseteq A_{\varepsilon}}\quad\text{for }A,B\in\mathscr{H}(\R^2),\]
where $A_{\varepsilon}:=\{x\in\R^2~\vert~d(x,A)<\varepsilon\}$ is the $\varepsilon-$\textit{neighborhood} of $A$.

\medskip

It is known from~\cite[2.10.21]{Federer69} that the distance function $h$ defines a metric on $\mathscr{H}(\R^2)$ and $(\mathscr{H}(\R^2),h)$ is a complete metric space.

\medskip

We consider the points in $\R^2$
\begin{align*}
&p_1:=(0,0), & &p_2:=\left(\frac{1}{2},\frac{\sqrt{3}}{2}\right),& &p_3:=(1,0),\\
&p_4:=\frac{p_2+p_3}{2},& &p_5:=\frac{p_1+p_3}{2},& &p_6:=\frac{p_1+p_2}{2}.
\end{align*}

Note that $p_1,p_2,p_3$ are the vertices of an equilateral triangle of side length $1$. 

\medskip

For any fixed $\alpha\in (0,1/3)$ we define the contractions
\begin{align*}
\Gai\colon&\R^2\longrightarrow\hspace*{0.3cm}\R^2&\nonumber\\
& x\hspace*{0.2cm}\longmapsto A_i(x-p_i)+p_i&i=1,\ldots, 6,
\end{align*}
where $A_1=A_2=A_3=\frac{1-\alpha}{2}\operatorname{I_2}$ and
\[ 
A_4=\frac{\alpha}{4}\left( \begin{array}{lr}
                           1&-\sqrt{3}\\
			   -\sqrt{3}&3
                           \end{array}\right),\quad
A_5=\alpha\left( \begin{array}{lr}
                            1&0\\
			    0&0
                            \end{array}\right),\quad
A_6=\frac{\alpha}{4}\left( \begin{array}{lr}
                            1&\sqrt{3}\\
			    \sqrt{3}&3
                            \end{array}\right).
\]                    

\medskip
      
It follows from~\cite{Hut81} that there exists a unique $\Ka\in\mathscr{H}(\R^2)$ such that 
\[\Ka=\bigcup_{i=1}^6G_{\alpha,i}(\Ka).\]
This set is called the \textit{Hanoi attractor of parameter} $\alpha$ and it is not self-similar because $\Ga{4},\Ga{5}$ and $\Ga{6}$ are not similitudes. The quantity $\alpha$ should be understood as the length of the segments joining the copies $\Ga{1}(\Ka)$, $\Ga{2}(\Ka)$ and $\Ga{3}(\Ka)$. The lack of self-similarity carries some difficulties that we discuss later.

\begin{figure}[h!tpb]
\centering
\includegraphics[scale=0.15]{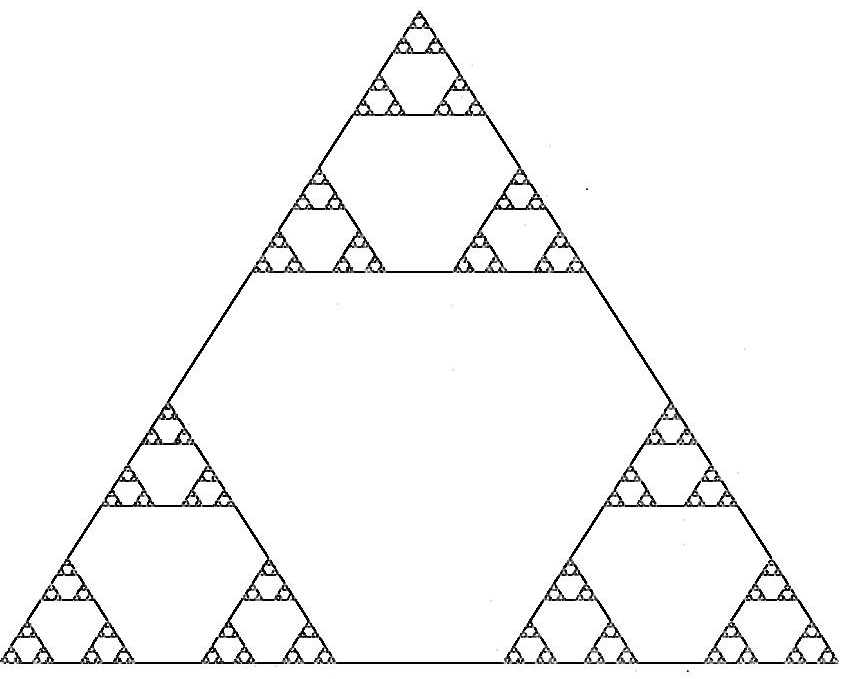}
\caption{The Hanoi attractor of parameter $\alpha=0.25$.}
\end{figure}

For the rest of this section, we fix $\alpha\in (0,1/3)$ and denote by $\A$ the alphabet consisting of the three symbols $1,2$ and $3$. For any word $w=w_1\cdots w_n\in\An$ of length $n\geq 1$, we define $\Gaw\colon\R^2\to\R^2$ as
\[\Gaw:=\Ga{w_1}\circ G_{\alpha,w_2}\circ\cdots \circ G_{\alpha,w_n}\]
and $G_{\alpha,\text{\o}}:=\operatorname{id}_{\R^2}$ for the empty word $\emptyset$.

\medskip

We will approximate the Hanoi attractor $\Ka$ by a sequence of one-dimensional sets defined as follows:

\medskip

Firstly, we consider for each $n\in\N_0$ the set
\begin{equation}\label{Def Wan}
\Wan :=\bigcup_{w\in\An}\Ga{w}(\{p_1,p_2,p_3\}).
\end{equation}

Secondly, we define $\Jak{0}:=\emptyset$ and 
\begin{equation}\label{Def Jan}
\Jan:=\bigcup_{m=0}^{n-1}\bigcup_{w\in\A^m}\Ga{w}\left(\bigcup_{i=1}^3e_i\right)
\end{equation}
for each $n\geq 1$, where $e_i$ denotes the line segment joining $G_{\alpha,j}(p_k)$ and $G_{\alpha,k}(p_j)$ for each triple $\{i,j,k\}=\mathcal{A}$ without its endpoints, as shown in Figure~\ref{Def_ei}. Note that $e_i$ has precisely length $\alpha$ for all $i=1,2,3$.

\medskip

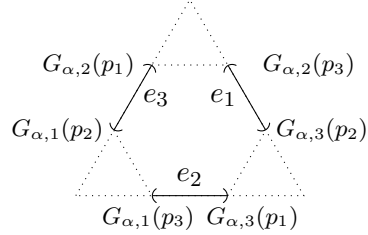
\begin{figure}[h!tpb]
\centering
\begin{tikzpicture}[scale=0.5]
\coordinate (p_1) at (0,0);
\fill[white] (p_1) circle (3pt);
\coordinate (p_2) at (3,5.1961);
\fill[white] (p_2) circle (3pt);
\coordinate (p_3) at (6,0);
\fill[white] (p_3) circle (3pt);
\coordinate [label=left:\footnotesize{$\Ga{1}(p_2)$}](p_12) at (1,1.732);
\coordinate [label=left:\footnotesize{$\Ga{2}(p_1)\;$}](p_21) at (2,3.4641);
\coordinate [label=below:\footnotesize{$\Ga{1}(p_3)\;$}](p_13) at (2,0);
\coordinate [label=below:\footnotesize{$\qquad\Ga{3}(p_1)$}](p_31) at (4,0);
\coordinate [label=right:\footnotesize{$\quad\Ga{2}(p_3)$}](p_23) at (4,3.4641);
\coordinate [label=right:\footnotesize{$\Ga{3}(p_2)$}](p_32) at (5,1.732);
\draw[dotted] (p_1) -- (p_2) -- (p_3) -- cycle (p_12) -- (p_13) (p_21) -- (p_23) (p_31) -- (p_32);
\draw[(-)] (p_12) -- (p_21) node[midway, right] {$e_3$};
\draw[(-)] (p_13) -- (p_31) node[midway, above] {$e_2$};
\draw[(-)] (p_23) -- (p_32) node[midway, left] {$e_1$};
\end{tikzpicture}
\caption{\small The set $J_{\alpha,1}$.}
\label{Def_ei}
\end{figure}

\medskip

Therefore, $\Jan$ denotes the set of line segments joining the copies of $\Ka$ at iteration level $n$.

\medskip

Finally we define for each $n\in\N_0$ the set
\begin{equation}\label{Def Van}
\Van:=\Wan\cup\Jan
\end{equation}
Since the sequence $(\Van)_{n\in\N_0}$ is monotonically in\-crea\-sing as suggested in Figure~\ref{Vans}, we can consider the set
\begin{equation}\label{eq def Vas}
\Vas:=\bigcup\limits_{n\in\N_0}\Van,
\end{equation}
which is dense in $\Ka$ with respect to the Euclidean metric (see~\cite[Lemma 2.1.2]{ARThesis13} for a proof). We may also say that $\Van$ is the union of a ``discrete part'' $\Wan$ and its ``continuous part''$\Jan$. Moreover, since $\Va{0}=\{p_1,p_2,p_3\}$ is independent of $\alpha$, we will denote this set just by $V_0$.

\bigskip

\begin{figure}[h!tpb]
\begin{center}
\begin{tabular}{ccccccc}
\begin{tikzpicture}[scale=0.375]
\coordinate (p_1) at (0,0);
\fill (p_1) circle (3.5pt);
\coordinate (p_2) at (3,5.1961);
\fill (p_2) circle (3.5pt);
\coordinate (p_3) at (6,0);
\fill (p_3) circle (3.5pt);
\end{tikzpicture}
&
&
\begin{tikzpicture}[scale=0.375]
\coordinate (p_1) at (0,0);
\fill (p_1) circle (3.5pt);
\coordinate (p_2) at (3,5.1961);
\fill (p_2) circle (3.5pt);
\coordinate (p_3) at (6,0);
\fill (p_3) circle (3.5pt);
\coordinate (p_12) at (1,1.732);
\fill (p_12) circle (3.5pt);
\coordinate (p_21) at (2,3.4641);
\fill (p_21) circle (3.5pt);
\coordinate (p_13) at (2,0);
\fill (p_13) circle (3.5pt);
\coordinate (p_31) at (4,0);
\fill (p_31) circle (3.5pt);
\coordinate (p_23) at (4,3.4641);
\fill (p_23) circle (3.5pt);
\coordinate (p_32) at (5,1.732);
\fill (p_32) circle (3.5pt);

\draw (p_12) -- (p_21)  (p_13) -- (p_31)   (p_23) -- (p_32);

\end{tikzpicture}
&
&
\begin{tikzpicture}[scale= 0.375]
\coordinate (p_1) at (0,0);
\fill (p_1) circle (3.5pt);
\coordinate (p_2) at (3,5.1961);
\fill (p_2) circle (3.5pt);
\coordinate (p_3) at (6,0);
\fill (p_3) circle (3.5pt);
\coordinate (p_12) at (1,1.732);
\fill (p_12) circle (3.5pt);
\coordinate (p_21) at (2,3.4641);
\fill (p_21) circle (3.5pt);
\coordinate (p_13) at (2,0);
\fill (p_13) circle (3.5pt);
\coordinate (p_31) at (4,0);
\fill (p_31) circle (3.5pt);
\coordinate (p_23) at (4,3.4641);
\fill (p_23) circle (3.5pt);
\coordinate (p_32) at (5,1.732);
\fill (p_32) circle (3.5pt);
\coordinate (p_112) at (1/3,1.732/3);
\fill (p_112) circle (3.5pt);
\coordinate (p_121) at (2/3,3.4641/3);
\fill (p_121) circle (3.5pt);
\coordinate (p_113) at (2/3,0);
\fill (p_113) circle (3.5pt);
\coordinate (p_131) at (4/3,0);
\fill (p_131) circle (3.5pt);
\coordinate (p_123) at (4/3,3.4641/3);
\fill (p_123) circle (3.5pt);
\coordinate (p_132) at (5/3,1.732/3);
\fill (p_132) circle (3.5pt);
\coordinate (p_212) at (1/3+2,1.732/3+3.4641);
\fill (p_212) circle (3.5pt);
\coordinate (p_221) at (2/3+2,3.4641/3+3.4641);
\fill (p_221) circle (3.5pt);
\coordinate (p_213) at (2/3+2,3.4641);
\fill (p_213) circle (3.5pt);
\coordinate (p_231) at (4/3+2,3.4641);
\fill (p_231) circle (3.5pt);
\coordinate (p_223) at (4/3+2,3.4641/3+3.4641);
\fill (p_223) circle (3.5pt);
\coordinate (p_232) at (5/3+2,1.732/3+3.4641);
\fill (p_232) circle (3.5pt);
\coordinate (p_312) at (1/3+4,1.732/3);
\fill (p_312) circle (3.5pt);
\coordinate (p_321) at (2/3+4,3.4641/3);
\fill (p_321) circle (3.5pt);
\coordinate (p_313) at (2/3+4,0);
\fill (p_313) circle (3.5pt);
\coordinate (p_331) at (4/3+4,0);
\fill (p_331) circle (3.5pt);
\coordinate (p_323) at (4/3+4,3.4641/3);
\fill (p_323) circle (3.5pt);
\coordinate (p_332) at (5/3+4,1.732/3);
\fill (p_332) circle (3.5pt);

\draw (p_112) -- (p_121) (p_113) -- (p_131) (p_123) -- (p_132)
(p_212) -- (p_221) (p_213) -- (p_231) (p_223) -- (p_232) 
(p_312) -- (p_321) (p_331) -- (p_313) (p_323) -- (p_332)
(p_12) -- (p_21) (p_13) -- (p_31) (p_23) -- (p_32)  (p_223) -- (p_232);
\end{tikzpicture}
&
\begin{tikzpicture}[scale=0.375]
\coordinate (p_1) at (0,0);
\fill[white] (p_1) circle (3pt);
\coordinate (p_3) at (0,2.7);
\coordinate (p_4) at (1.5,2.7);
\draw[->] (p_3) -- (p_4);
\end{tikzpicture}
&
\includegraphics[width=2.4cm,height=2.1cm]{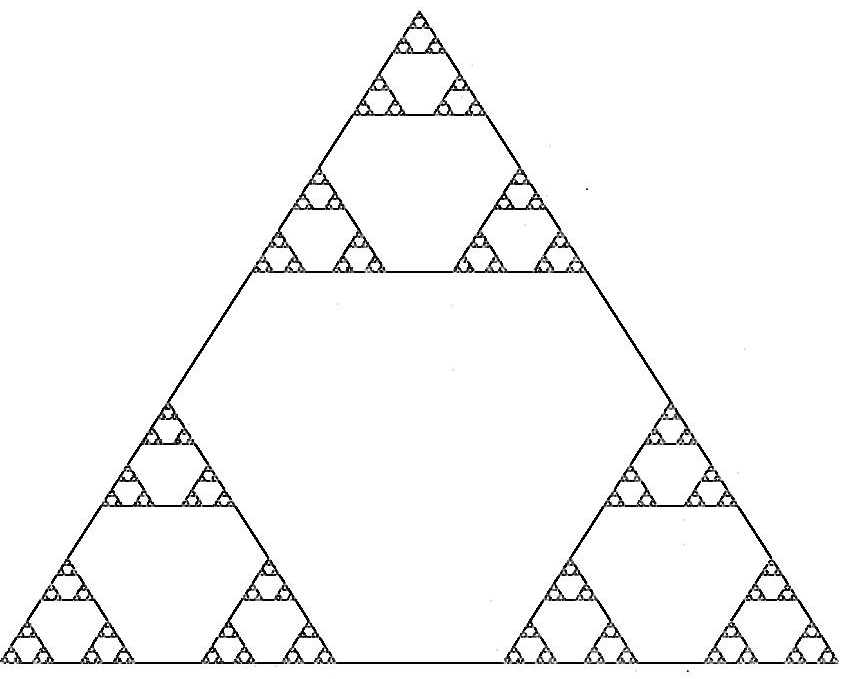}
\end{tabular}
\end{center}
\caption{\small $V_{0}$, $\Va{1}$, $\Va{2}$ and $\Ka$.}
\label{Vans}
\end{figure}

\medskip

The geometric relationship between Hanoi attractors and the Sierpi\'nski gasket is stated in the following Theorem.

\begin{thm}\label{thm: Properties of Ka}
Let $K$ denote the Sierpi\'nski gasket and let $\Ka$ be the Hanoi attractor of parameter $\alpha$, $\alpha\in (0,1/3)$. Then we have:
\begin{itemize}
\item[(i)] $h(K_{\alpha},K)\longrightarrow 0$ as $\alpha\downarrow 0$,
\item[(ii)] $\dim_H \Ka=\frac{\ln 3}{\ln 2-\ln (1-\alpha)}=:d$ and $0<\mathcal{H}^d(\Ka)<\infty$. In particular,
\[\dim_H \Ka\xrightarrow{\alpha\downarrow 0} \dim_H K.\]
\end{itemize}
\end{thm}
\begin{proof}
See~\cite[Theorem 3.1, Corollary 4.1]{AF12}.
\end{proof}

\medskip

\begin{rem}
Note that part \textit{(ii)} of this Theorem justifies the condition that $\alpha<1/3$: If $\alpha\geq 1/3$, then $\dim_H\Ka=1$, reducing the problem to $1-$dimensional analysis.
\end{rem}

These results awoke the question about what other convergence types could hold, in particular convergence of the spectral dimension. Since $\Ka$ is not self-similar, we could neither define a Dirichlet form for $\Ka$ nor calculate its spectral dimension as in the self-similar case treated in~\cite{KL93}. 
However, $\Ka$ still has the good property of being finitely ramified and so we focus on constructing a resistance form $(\EKa,\FKa)$ on $\Ka$. After choosing a suitable Radon measure, this resistance form induces a Dirichlet form on $\Ka$ and therefore a Laplacian, whose spectral behavior we analyse.

\medskip

This paper is organized as follows: Section 2 recalls the construction of the local and regular Dirichlet form $(\EKa,\Da)$ on $\Ka$ introduced in~\cite{AF13a} restating some of the results in terms of resistance forms. In particular, we prove that
\begin{thm}
There exists a regular resistance form $(\EKa,\FKa)$ on $\Ka$.
\end{thm}

Section 3 deals with the properties of the Dirichlet form $(\EKa,\Da)$ induced by a class of Radon measures in $\Ka$ that depend on a parameter $\beta$. We also characterize the spectrum of the Laplacian associated with $(\EKa,\Da)$ in the corresponding $L^2-$space. Section 4 analyses the asymptotic behavior of the eigenvalue counting function of this Laplacian by giving the following estimate

\begin{thm}
Let $N_{D/N}(x)$ denote the eigenvalue counting function of the Laplacian on $\Ka$. Then
\begin{equation}\label{asymp intro}
N_{N/D}(x)\asymp x^{\frac{\log 3}{\log 5}},\qquad x\to\infty.
\end{equation}
\end{thm}

The eigenvalue counting function gives the number of Neumann (resp. Dirichlet) eigenvalues of the considered Laplacian, counted with multiplicity, lying below $x$. A more precise definition is given at the beginning of Section 4.

\medskip

From this theorem it follows that the \textit{spectral dimension} of $\Ka$ equals $\frac{\log 9}{\log 5}$ for all $\alpha\in (0,1/3)$ and it therefore coincides with the spectral dimension of the Sierpi\'nski gasket. In particular, it turns out that (contrary to Hausdorff dimension) the spectral dimension of $\Ka$ is independent of the parameter $\alpha$. This can be interpreted as the fact that one can ``see'' this parameter but not ``hear'' it.
 However, the parameter $\alpha$ will be reflected by the constants appearing in the asymtotics~\eqref{asymp intro}, where we also provide a second term whose constants do not only depend of $\alpha$ but also of the measure parameter $\beta$, see Theorem~\ref{prop estimation evcf}. 
 
\medskip

The last section analyses some interesting physical consequences of this result in view of the Einstein relation to be considered for further research.

\section{Resistance and Dirichlet form on the Hanoi attractor}\label{sec: DF on Ka}
This section is devoted to the construction of a resistance  and a Dirichlet form . The main novelty in this procedure consists in the definition of the approximating forms $(\ean,\Dan)$, which combines techniques of discrete and quantum graphs, as well as the computation of the corresponding renormalization factors.

\medskip

Since all results presented in the paper hold for any $\alpha\in (0,1/3)$ we drop off the parameter $\alpha$ in the definitions for ease of reading, and only recall this dependence on $\alpha$ explicitly when needed. Thus we write $G_w$, $W_n$, $J_n$, $V_n$, $\E_n$ etc. instead of $\Ga{w}$, $\Wan$, $\Jan$, $\Van$ and $\Ean$.

\subsection{Approximating forms}
The definition of the bilinear form $(\E_n,\D_n)$ on each of the approximating sets $V_n$ defined in~\eqref{Def Van} will reflect the fact that the set $V_n$ can be decomposed into a ``discrete'' and a ``continuous'' part.

\medskip

We start with some useful notation. Concerning to the ``discrete part'', we say that any two vertices $x,y\in W_n$, are \textit{$n-$neighbors}, and write
\[x\stackrel{n}{\sim}y,\]
if and only if there exists a word $w\in\An$ of length $n\in\N_0$ such that $x,y\in G_w(V_0)$, i.e., both points are vertices of the same $n-$th level triangle $G_{w}(\{p_0,p_1,p_2\})$.  Figure~\ref{Nbhs2} illustrates this relation for the level $n=2$.

\medskip

\begin{figure}[h!tpb]
\centering
\begin{tikzpicture}[scale= 0.45]
\coordinate (p_1) at (0,0);
\fill (p_1) circle (2.5pt);
\coordinate (p_2) at (3,5.1961);
\fill (p_2) circle (2.5pt);
\coordinate (p_3) at (6,0);
\fill (p_3) circle (2.5pt);
\coordinate[label=left:\footnotesize{$z$}] (p_12) at (1,1.732);
\fill (p_12) circle (2.5pt);
\coordinate (p_21) at (2,3.4641);
\fill (p_21) circle (2.5pt);
\coordinate (p_13) at (2,0);
\fill (p_13) circle (2.5pt);
\coordinate (p_31) at (4,0);
\fill (p_31) circle (2.5pt);
\coordinate[label=right:\footnotesize{$x$}] (p_23) at (4,3.4641);
\fill (p_23) circle (2.5pt);
\coordinate (p_32) at (5,1.732);
\fill (p_32) circle (2.5pt);
\coordinate (p_112) at (1/3,1.732/3);
\fill (p_112) circle (2.5pt);
\coordinate[label=left:\footnotesize{$t$}] (p_121) at (2/3,3.4641/3);
\fill (p_121) circle (2.5pt);
\coordinate (p_113) at (2/3,0);
\fill (p_113) circle (2.5pt);
\coordinate (p_131) at (4/3,0);
\fill (p_131) circle (2.5pt);
\coordinate (p_123) at (4/3,3.4641/3);
\fill (p_123) circle (2.5pt);
\coordinate (p_132) at (5/3,1.732/3);
\fill (p_132) circle (2.5pt);
\coordinate (p_212) at (1/3+2,1.732/3+3.4641);
\fill (p_212) circle (2.5pt);
\coordinate (p_221) at (2/3+2,3.4641/3+3.4641);
\fill (p_221) circle (2.5pt);
\coordinate (p_213) at (2/3+2,3.4641);
\fill (p_213) circle (2.5pt);
\coordinate (p_231) at (4/3+2,3.4641);
\fill (p_231) circle (2.5pt);
\coordinate (p_223) at (4/3+2,3.4641/3+3.4641);
\fill (p_223) circle (2.5pt);
\coordinate[label=right:\footnotesize{$y$}] (p_232) at (5/3+2,1.732/3+3.4641);
\fill (p_232) circle (2.5pt);
\coordinate (p_312) at (1/3+4,1.732/3);
\fill (p_312) circle (2.5pt);
\coordinate (p_321) at (2/3+4,3.4641/3);
\fill (p_321) circle (2.5pt);
\coordinate (p_313) at (2/3+4,0);
\fill (p_313) circle (2.5pt);
\coordinate (p_331) at (4/3+4,0);
\fill (p_331) circle (2.5pt);
\coordinate (p_323) at (4/3+4,3.4641/3);
\fill (p_323) circle (2.5pt);
\coordinate (p_332) at (5/3+4,1.732/3);
\fill (p_332) circle (2.5pt);

\draw[dotted] (p_1) -- (p_112) -- (p_113) -- cycle (p_121) -- (p_12) -- (p_123) -- cycle (p_131) -- (p_132) -- (p_13) -- cycle (p_21) -- (p_212) -- (p_213) -- cycle (p_2) -- (p_221) -- (p_223) -- cycle (p_231) -- (p_232) -- (p_23) -- cycle (p_31) -- (p_312) -- (p_313) -- cycle (p_321) -- (p_32) -- (p_323) -- cycle (p_331) -- (p_332) -- (p_3) -- cycle;
\end{tikzpicture}
\caption{\small Examples of $2-$neighbors: $x\stackrel{2}{\sim}y$ and $z\stackrel{2}{\sim}t$.}
\label{Nbhs2}
\end{figure}

Concerning to the ``continuous'' part, we define the set of line segments
\[\J_n:=\{e~\vert~e \text{ is a connected component of } J_n\}.\]
If necessary, we will specify the endpoints of such a component by writing $e=:(a_e,b_e)$. Note that $a_e,b_e\in W_n$ for $e\in\J_n$. Moreover, we denote by $H^1(e,dx):=\{f\circ\varphi_e~\vert~f\in H^1((0,1),dx)\}$, $H^1((0,1),dx)$ the classical Sobolev space of functions defined on the unit interval and $\varphi_e$ as defined in~\eqref{eq: def varphi_e}.

\medskip

\begin{defn}\label{defn: Def ean}
Let $\D_{0}:=\{u\colon V_0\to\R\}$ and
\[\D_n:=\{u\colon V_n\to\R~\vert~u_{\vert_e}\in H^1(e, dx)~\forall\,e\in\J_n\}\]
for each $n\in\N$. We define the quadratic form $E_n\colon \D_n\to\R$ by
\begin{equation*}
E_n[u]:=\sum_{x\stackrel{n}{\sim}y}(u(x)-u(y))^2+\int_{J_n}\abs{\nabla u}^2dx.
\end{equation*}
For each $u\in\D_n$, $E_n[u]$ is called the \textit{energy of} $u$ \textit{at level} $n$. 
\end{defn}
Moreover we can write $E_n[u]=E_n^d[u]+E_n^c[u]$, where $E_n^d , E_n^c\colon\D_n\to\R$ are defined by
\begin{equation}\label{equation: Def eandisc}
E_n^d[u]:=\sum_{x\stackrel{n}{\sim}y}(u(x)-u(y))^2
\end{equation}
and 
\begin{equation}\label{equation: Def eancont}
E_n^c[u]:=\int_{J_n}\abs{\nabla u}^2dx.
\end{equation}
We call these quadratic forms the \textit{discrete} and resp. \textit{continuous part} of $E_n$.

The integral expression in~\eqref{equation: Def eancont} has to be understood as follows: for each line segment $(a_e,b_e)\in \J_n$ we consider $\varphi_e\colon [0,1]\to\R^2$ to be the curve parametrization of $e$, that is
\begin{equation}\label{eq: def varphi_e}
\varphi_e(t):=(b_e-a_e)\cdot t+a_e.
\end{equation}
For any function $u\in\D_n$,
\begin{align*}\label{eq: Integral Enc}
E_n^c[u]&=\int_{J_n}\abs{\nabla u}^2dx:=\sum_{e\in\J_n}\frac{1}{b_e-a_e}\int_0^1\abs{(u\circ\varphi_e)'}^2\,dt.
\end{align*}

\medskip

Applying the polarization identity to this energy functional we obtain the bilinear form 
\begin{equation*}
E_n(u,v):=\frac{1}{2}\left(E_n[u+v]-E_n[u]-E_n[v]\right),\qquad u,v\in\D_n.
\end{equation*}

\subsection{Harmonic extension and renormalization factor}
So far we have defined $\ean$ just by ``gluing'' its discrete and continuous part, $E_n^d$ and $E_n^c$. This means that, until now, both parts of the energy are independent of each other. However, since we want the energy functional $E_n$ to become a resistance form, we need it to be invariant under harmonic extension. Thus we still have to renormalize it. This renormalization is precisely what correlates $E_n^d$ and $E_n^c$.

\subsubsection{Harmonic extension}
In this paragraph we explain how to construct the harmonic extension of any function $u\colon V_0\to\R$ to any level $n\geq1$.

\begin{defn} 
Let $u\in\D_0$. Its \textit{harmonic extension to level} $1$ is the function $\tilde{u}\in\D_{1}$ satisfying
\begin{equation*}\label{equation: harmExtStepa}
E_{1}[\tilde{u}]=\inf\{E_{1}[v]~\vert~v\in\D_1\text{  and  }v_{\vert_{V_0}}\equiv u\}.
\end{equation*}
\end{defn}

This extension is well defined, as the next proposition shows.

\medskip

\begin{prop}\label{prop: HarmExtLevel1}
For any function $u\in\D_0$,
\begin{equation}\label{eq: min problem}
\inf\{E_{1}[v]~\vert~v\in\D_{1}~\text{and}~v_{\vert_{V_0}}\equiv u\}
\end{equation}
is attained by a unique function $\tilde{u}\in\D_{1}$ defined on $W_{1}$ by
\begin{equation}\label{equation: solHarm1gen}
\tilde{u}_1(G_i(p_j))=\frac{2+3\alpha}{5+3\alpha}u(p_i)+\frac{2}{5+3\alpha}u(p_j)+\frac{1}{5+3\alpha}u(p_k)
\end{equation}
for any $i\in\A$, $\{i,j,k\}=\A$, and linear interpolation on $J_{1}$.
\end{prop}

\begin{proof}
Without loss of generality, we may assume that the function $u_0\in\D_0$ is given by
\[u_0 (p_1)=1, \quad u_0(p_2)=0=u_0(p_3).\]
If we know the values of the extension $\tilde{u}_1$ on $W_{1}$, then energy is minimized by extending the function $\tilde{u}_1\vert_{W_{1}}$ linearly to $J_{1}$, i.e.
\begin{equation*}
\tilde{u}_1\vert_{e}(x):=\frac{\tilde{u}_1(b_e)-\tilde{u}_1(a_e)}{b_e-a_e}\cdot x
+\frac{\tilde{u}_1(a_e)b_e-\tilde{u}_1(b_e)a_e}{b_e-a_e}
\end{equation*}
for each $x\in (a_e,b_e)=(G_i(p_j),G_j(p_i))\subseteq J_{1}$, $i\neq j$ (see Figure~\ref{harmExt1Formula}). 

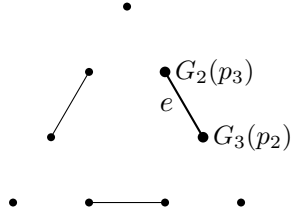
\begin{figure}[h!tpb]
\centering
\begin{tikzpicture}[scale=0.5]
\coordinate (p_11) at (0,0);
\fill (p_11) circle (3pt);
\coordinate (p_22) at (3,5.1961);
\fill (p_22) circle (3pt);
\coordinate (p_33) at (6,0);
\fill (p_33) circle (3pt);
\coordinate (p_12) at (1,1.732);
\fill (p_12) circle (3pt);
\coordinate (p_21) at (2,3.4641);
\fill (p_21) circle (3pt);
\coordinate (p_13) at (2,0);
\fill (p_13) circle (3pt);
\coordinate (p_31) at (4,0);
\fill (p_31) circle (3pt);
\coordinate [label=right:\small{$G_{2}(p_3)$}](p_23) at (4,3.4641);
\fill (p_23) circle (4pt);
\coordinate [label=right:\small{$G_{3}(p_2)$}](p_32) at (5,1.732);
\fill (p_32) circle (4pt);
\draw (p_12) -- (p_21)  (p_13) -- (p_31);
\draw[thick] (p_23) -- (p_32) node[midway, left] {$e$};

\end{tikzpicture}
\caption{\small{Harmonic extension $\tilde{u}_1$.}}
\label{harmExt1Formula}
\end{figure}
The integrals of the continuous part of the energy thus become
\[\int_{e}\abs{\nabla \tilde{u}_1}^2 dx=\frac{(\tilde{u}_1(b_e)-\tilde{u}_1(a_e))^2}{\abs{b_e-a_e}}\]
and the total energy $E_{1}[\tilde{u}_1]$ can be expressed only in terms of ${W_1}$.

Due to the definition of $u_0$ and the symmetry of $V_{1}$, the function $\tilde{u}_1$ on $W_{1}$ will have the unknown values $x,y$ and $z$ as shown in Figure~\ref{harmExt1}.

\enlargethispage{-0.5cm}
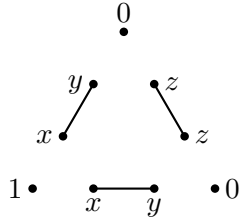
\begin{figure}[h!tpb]
\centering
\begin{tikzpicture}[scale=0.4]
\coordinate [label=left:$1$] (p_11) at (0,0);
\fill (p_11) circle (4pt);
\coordinate [label=above:$0$] (p_22) at (3,5.1961);
\fill (p_22) circle (4pt);
\coordinate [label=right:$0$] (p_33) at (6,0);
\fill (p_33) circle (4pt);
\coordinate (p_4) at (1.5,2.598);
\coordinate (p_5) at (3,0);
\coordinate (p_6) at (4.5,2.598);
\coordinate [label=left:$x$](p_12) at (1,1.732);
\fill (p_12) circle (4pt);
\coordinate [label=left:$y$](p_21) at (2,3.4641);
\fill (p_21) circle (4pt);
\coordinate [label=below:$x$](p_13) at (2,0);
\fill (p_13) circle (4pt);
\coordinate [label=below:$y$](p_31) at (4,0);
\fill (p_31) circle (4pt);
\coordinate [label=right:$z$](p_23) at (4,3.4641);
\fill (p_23) circle (4pt);
\coordinate [label=right:$z$](p_32) at (5,1.732);
\fill (p_32) circle (4pt);
\draw[thick] (p_12) -- (p_21)  (p_13) -- (p_31)   (p_23) -- (p_32);

\end{tikzpicture}
\caption{\small{Values of $\tilde{u}_1$ in $W_1$.}}
\label{harmExt1}
\end{figure}

Let us now define the so--called \textit{conductance} of an edge $\{p,q\}$ by
\begin{equation*}
c^{1}_{pq}:=\left\{\begin{array}{rl}
									1,& \text{if }p\stackrel{1}{\sim}q,\\
									\alpha^{-1},&\text{if }(p,q)=:e\in \J_{1}.
\end{array}\right.
\end{equation*}
The energy of the harmonic extension $\tilde{u}_1$ can be thus expressed as the sum
\begin{equation*}
E_{1}[\tilde{u}_1]=\frac{1}{2}\sum_{p,q\in W_{1}}c^{1}_{pq}(\tilde{u}_1(p)-\tilde{u}_1(q))^2.
\end{equation*}
Solving the minimization problem in~\eqref{eq: min problem} leads to a linear system of equations whose solution is given by
\begin{equation}\label{equation: solHarm1}
x=\frac{2+3\alpha}{5+3\alpha},\quad y=\frac{2}{5+3\alpha},\quad z=\frac{1}{5+3\alpha}.
\end{equation}

Because of symmetry and linearity, given an arbitrary function $u_0\colon V_0\to\R$ with
\[u_0 (p_1)=a, \quad u_0(p_2)=b,\quad u_0(p_3)=c,\qquad a,b,c\in\R,\]
the harmonic extension $\tilde{u}_1$ is given by
\begin{equation*}
\tilde{u}_1(p)=\frac{2+3\alpha}{5+3\alpha}a+\frac{2}{5+3\alpha}b+\frac{1}{5+3\alpha}c
\end{equation*}
for a point $p$ as in Figure~\ref{generalizedharmExt1}
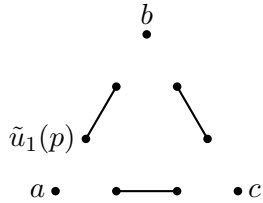
\begin{figure}[h!tpb]
\centering
\begin{tikzpicture}[scale=0.4]
\coordinate [label=left:$a$] (p_11) at (0,0);
\fill (p_11) circle (4pt);
\coordinate [label=above:$b$] (p_22) at (3,5.1961);
\fill (p_22) circle (4pt);
\coordinate [label=right:$c$] (p_33) at (6,0);
\fill (p_33) circle (4pt);
\coordinate (p_4) at (1.5,2.598);
\coordinate (p_5) at (3,0);
\coordinate (p_6) at (4.5,2.598);
\coordinate [label=left:$\tilde{u}_1(p)$](p_12) at (1,1.732);
\fill (p_12) circle (4pt);
\coordinate (p_21) at (2,3.4641);
\fill (p_21) circle (4pt);
\coordinate (p_13) at (2,0);
\fill (p_13) circle (4pt);
\coordinate (p_31) at (4,0);
\fill (p_31) circle (4pt);
\coordinate (p_23) at (4,3.4641);
\fill (p_23) circle (4pt);
\coordinate (p_32) at (5,1.732);
\fill (p_32) circle (4pt);
\draw[thick] (p_12) -- (p_21)  (p_13) -- (p_31)   (p_23) -- (p_32);

\end{tikzpicture}
\caption{\small{The extension $\tilde{u}_1$ at $p\in W_{\alpha,1}$ for an arbitrary $u_0$.}}
\label{generalizedharmExt1}
\end{figure}

The uniqueness of the extension is given by the uniqueness of the solution of the linear system corresponding to the minimization problem.
\end{proof}

The expression given in~\eqref{equation: solHarm1gen} may be considered as a kind of ``extension algorithm'', where $\alpha$ is the length of the segment lines in $\J_{1}$.

Next proposition generalizes this last argument in order to construct the \textit{harmonic extension from any level $n$ to $n+1$}.
\begin{prop}\label{prop: harmExtLeveln}
Let $d_{0}:=0$ and $d_{n}:=\alpha\left(\frac{1-\alpha}{2}\right)^{n-1}$ for each $n\in\N$.

For any function $u\in\D_n$, the infimum
\[\inf\{E_{n+1}[v]~\vert~v\in\D_{n+1}\text{  and  }v\vert_{V_n}\equiv u\}\]
is attained by a unique function $\tilde{u}\in\D_{n+1}$ which is given at each $p_{wij}:=G_{wi}(p_j)\in W_{n+1}$ by
\begin{equation*}
\tilde{u}( p_{wij})=\frac{2+3d_{n}}{5+3d_{n}}u(p_{wii})+\frac{2}{5+3d_{n}}u(p_{wjj})+\frac{1}{5+3d_{n}}u(p_{wkk})
\end{equation*}
for $wi\in\A^{n+1}$, $\{i,j,k\}=\A$, and linear interpolation on $J_{n+1}\setminus J_n$.
\end{prop}

\begin{proof}
We define the conductance of the edges $\{p,q\}$ for $p,q\in W_n$ by
\begin{equation}\label{eq def c_ns}
c^{n}_{pq}:=\left\{\begin{array}{rl}
								1&\text{if } p\stackrel{n}{\sim}q,\\
								d_{n}^{-1}&\text{if }(p,q)=:e\in\J_{n}\setminus\J_{n-1}.
\end{array}\right.
\end{equation}

Due to finite ramification and recursive structure of $\Ka$ the proof works entirely analogous to Proposition~\ref{prop: HarmExtLevel1} (see~\cite[Section 2.2]{AF13a} for details).
\end{proof}

Iterating the last proposition leads to the following definition:

\begin{defn} 
Let $u\in\D_0$. Its \textit{harmonic extension to level} $n$ is the unique function $\tilde{u}\in\D_{n}$ satisfying
\begin{equation*}\label{equation: harmExtStepa}
E_{n}[\tilde{u}]=\inf\{E_{n}[v]~\vert~v\in\D_{n}\text{  and  }v_{\vert_{V_0}}\equiv u\}.
\end{equation*}
\end{defn}

\subsubsection{Renormalization}
Let $\tilde{u}\in\D_{n+1}$ denote the harmonic extension of a function $u\in\D_n$. A sequence of bilinear forms $\{B_n\colon\D_n\times\D_n\to\R\}_{n\in\N_0}$ is said to be \textit{invariant under harmonic extension} if
\[
B_n(u,u)=B_{n+1}(\tilde{u},\tilde{u})\qquad\text{for all }u\in\D_n.
\]

\medskip

If we can find a sequence of positive numbers $(\rho_n)_{n\in\N_0}$ such that the sequence of bilinear forms $\{\E_n\}_{n\in\N_0}$ defined by
\[\E_n(u,u):=\rho_n^{-1}E_n(u,u)\]
is invariant under harmonic extension, then $\rho_n$ is called the \textit{renormalization factor} of $E_n$ for each $n\in\N_0$. The aim of this section is the computation of this factor. Contrary to the typical self-similar case, we now have different quantities $\rho_n^d$, $\rho_n^c$ for the discrete and continuous energy. We will see, that $\rho_n^d$ and $\rho_n^c$ are not independent from each other, and hence they ``glue'' together both $E_n^d$ and $E_n^c$.

For each $n\geq 1$ we define

\begin{equation}\label{def r_ns}
r_n^d:=\frac{3}{5+3d_n}\qquad r_n^c:=\frac{3d_n}{5+3d_n},
\end{equation}
where $d_n$ was defined in Proposition~\ref{prop: harmExtLeveln}. Set $\rho_0^d:=1$ and define for each $n\geq 1$ the numbers
\begin{equation}\label{def rho_ns}
\rho_n^d:=\prod_{i=1}^n r_i^d\qquad\rho_n^c:=\rho_{n-1}^d\cdot r_n^c,
\end{equation}
with $r_i^d$, $r_n^c$ as in~\eqref{def r_ns}.

For each $n\geq 1$, define the quadratic form $\E_n\colon\D_n\to\R$ by
\begin{equation}\label{def ren E_n}
\E_n[u]:=\frac{1}{\rho_n^d}E^d_n[u]+\sum_{k=1}^n\frac{1}{\rho_k^c}E^c_{k^{-}}[u],
\end{equation}
where $E^c_{k^-}[u]:=\sum\limits_{e\in\J_n\setminus\J_{n-1}}\int_0^1\abs{(u\circ\varphi_e)'}^2\,dt$. 

We will also denote by $\E_n$ the bilinear form obtained from the polarization identity
\begin{equation*}
\E_n(u,v):=\frac{1}{2}\left(\E_n[u+v]-\E_n[u]-\E_n[v]\right), \qquad u,v\in\D_n.
\end{equation*}

The following proposition ensures that the renormalized forms $\E_n$ are invariant under harmonic extension.

\begin{prop}\label{prop renorm}
Let $\tilde{u}_n\colon V_n\to\R$ be the harmonic extension to level $n\geq 1$ of a function $u_0\colon V_0\to\R$, then
\begin{equation}\label{eq prop renorm}
\E_0[u_0]=E_0[u_0]=\frac{1}{\rho_n^d}E^d_n[\tilde{u}_n]+\sum_{k=1}^n\frac{1}{\rho_k^c}E^c_{k^-}[\tilde{u}_n]=\E_n[\tilde{u}_n].
\end{equation}
\end{prop}
\begin{proof}
We apply the $\Delta- Y$ transform to an $n-$cell with the resistances (inverse conductances) given in~\eqref{eq def c_ns} as Figure~\ref{DeltaYTrafo} shows.

\begin{figure}[h]
\centering
\begin{tabular}{cccc}
\begin{tikzpicture}[scale=0.45]
\coordinate (A) at (0,-0.5);
\fill [white] (A) circle (0.1pt);
\coordinate (p_11) at (0,0);
\coordinate (p_22) at (3,5.1961);
\coordinate (p_33) at (6,0);
\coordinate (p_12) at (1,1.732);
\coordinate (p_21) at (2,3.4641);
\coordinate (p_13) at (2,0);
\coordinate (p_31) at (4,0);
\coordinate (p_23) at (4,3.4641);
\coordinate (p_32) at (5,1.732);
\draw (p_33) -- (p_11) -- (p_22);
\draw (p_22) -- (p_23) node[midway, right] {$1$};
\draw (p_23) -- (p_32) node[midway, right] {$d_n$};
\draw (p_23) -- (p_33) (p_12) -- (p_13) (p_21) -- (p_23) (p_31) -- (p_32);
\end{tikzpicture}
&
\begin{tikzpicture}[scale=0.5]
\coordinate (A) at (0,0);
\fill [white] (A) circle (0.1pt);
\coordinate (B) at (0,2.5);
\coordinate (E) at (2,2.5);
\draw[->] (B) -- (E);
\end{tikzpicture}
&
\begin{tikzpicture}[scale=0.3]
\coordinate (p_11) at (0,0);
\coordinate (p_01) at (-2,-2);
\coordinate (p_22) at (3,5.1961);
\coordinate (p_02) at (3,7.5);
\coordinate (p_33) at (6,0);
\coordinate (p_03) at (8,-2);
\coordinate (p_1) at (3,5.75);
\coordinate (p_2) at (3,5.75);
\coordinate (p_2) at (3,5.75);
\draw (p_33) -- (p_11) -- (p_22);
\draw (p_22) -- (p_33) node[midway, right] {$\frac{2}{3}+d_n$};
\draw (p_01) -- (p_11) (p_03) -- (p_33);
\draw (p_02) -- (p_22) node[midway, right] {\small{$\frac{1}{3}$}};
\end{tikzpicture}
&
\\
& & & 
\begin{tikzpicture}[scale=0.3]
\draw[->] (0,0) arc (90:-90:1cm);
\end{tikzpicture}
\\
\begin{tikzpicture}[scale=0.3]
\coordinate (A) at (0,-1);
\fill [white] (A) circle (0.1pt);
\coordinate (p_11) at (0,0);
\coordinate (p_22) at (3,5.1961);
\coordinate (p_33) at (6,0);
\draw (p_33) -- (p_11) -- (p_22);
\draw (p_22) -- (p_33) node[midway, right] {$\frac{5+3d_n}{3}$};
\end{tikzpicture}
&
\begin{tikzpicture}[scale=0.5]
\coordinate (A) at (0,0);
\fill [white] (A) circle (0.1pt);
\coordinate (B) at (0,2);
\coordinate (E) at (2,2);
\draw[<-] (B) -- (E);
\end{tikzpicture}
&
\begin{tikzpicture}[scale=0.4]
\coordinate (p_11) at (0,0);
\coordinate (p_22) at (3,5.1961);
\coordinate (p_33) at (6,0);
\coordinate (p_c) at (3,2.5);
\draw (p_11) -- (p_c) (p_33) -- (p_c);
\draw (p_22) -- (p_c) node[midway, right] {$\frac{5}{9}+\frac{d_n}{3}$};
\end{tikzpicture}
\end{tabular}
\caption{\small{$\Delta-Y$ transform for an $n-$cell.}}
\label{DeltaYTrafo}
\end{figure}
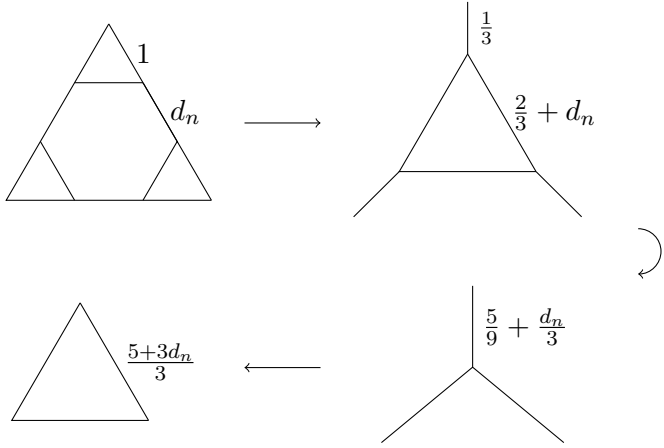

Since we are looking for electrical equivalence to a triangular network with wire resistance one, we get that the correct resistances at level $n\geq 1$ are $r^d_n=\frac{3}{5+3d_n}$ for the discrete part, and $r_n^c=\frac{3d_n}{5+3d_n}$ for the continuous part. Iterating this procedure and comparing the outcome with~\eqref{def r_ns} and the definition of the renormalization factors $\rho_n^d$ and $\rho_n^c$ in~\eqref{def rho_ns} proves~\eqref{eq prop renorm}.
\end{proof}

\subsection{Resistance form}
We first recall the definition of resistance form on a locally compact metric space $(X,d)$. We refer to~\cite{Kig12} for an outline of the most important results on the theory of resistance forms. 
\begin{defn}\label{def: resform}
The pair $(E,F)$ is called a \textit{resistance form} if the following properties are satisfied:
\begin{itemize}
\item[(R1)] $F$ is a linear subspace of $\{u\colon X\to\R\}$ that contains constants. Moreover, $E$ is a non-negative symmetric bilinear form on $F$ and for all $u\in F$, $E(u,u)=0$ if and only if $u$ is constant.
\item[(R2)] For any $u,v\in F$, write $u\sim v$ if and only if $u-v$ is constant. Then $(F_{/\sim},E)$ is a complete metric space.
\item[(R3)] For any two points $x,y\in X$, there exists $u\in F$ such that $u(x)\neq u(y)$.
\item[(R4)] For any two points $x,y\in X$,
\[
R(x,y):=\sup\left\{\frac{\abs{u(x)-u(y)}^2}{E(u,u)}~\vert~ u\neq 0,\, u\in F\right\}<\infty\quad\forall\,x,y\in X.
\]
This function $R\colon X\to\R_+$ defines a distance in $X$, which is called the \textit{resistance metric} associated with $(E,F)$.
\item[(R5)] For any $u\in F$, $\overline{u}:=0\vee u \wedge 1 \in F$ and $E(\overline{u},\overline{u})\leq E(u,u)$.
\end{itemize} 
\end{defn}

In order to construct our desired resistance form, we first define 
\begin{align*}\label{eq def D_ast}
\D_\ast&:=\{u\colon V_\ast\to\R~\vert~u_{\vert_{V_n}}\in\D_n,\;\forall\,n\in\N\text{ and }\lim_{n\to\infty}\E_n[u_{\vert_{V_n}}]<\infty\},\\
\E[u]&:=\lim_{n\to\infty}\E_n[u_{\vert_{V_n}}],\qquad u\in\D_\ast.
\end{align*}
In the next proposition we show that any function in $\D_\ast$ is H\"older -- and therefore uniformly -- continuous on $V_\ast$. Since $V_\ast$ is dense in $\Ka$, $u$ can be uniquely extended to a  function on $\Ka$.

\begin{lemma}\label{lem Dast in cont}
Every function in $\D_\ast$ is continuous on $\Ka$ with respect to the Euclidean metric.
\end{lemma}
\begin{proof}
Since $V_\ast$ is dense in $\Ka$ with respect to the Euclidean norm, it suffices to show continuity on $V_\ast$. Consider $u\in\D_\ast$ and $x,y\in V_\ast$. 
\begin{enumerate}
\item[(1)] If $x,y\in W_n$ are $n$-neighbors, then $\abs{x-y}=\left(\frac{1-\alpha}{2}\right)^n$ and 
\[
\frac{1}{\rho_n^d}\abs{u(x)-u(y)}^2\leq \E_n^d[u]\leq\E[u],
\]
which implies that
\begin{align*}
\abs{u(x)-u(y)}&\leq\left(\frac{1}{\rho_n^d}\right)^{1/2}\E^{1/2}[u]\leq 
\E^{1/2}[u]\abs{x-y}^{l_{\alpha}},
\end{align*}
where $l_{\alpha}:=\frac{\ln 3-\ln 5}{2(\ln(1-\alpha)-\ln 2)}$.
\item[(2)] If $x,y$ belong to the same component $e\in\J_n$ for some $n\in\N$, $u$ is in particular continuous on $e$ so we get by Cauchy-Schwartz that
\[
\abs{u(x)-u(y)}^2=\abs{\int_x^y\nabla u\,dx}^2
\leq\int_e\abs{\nabla u}^2dx\cdot\abs{x-y},
\]
and therefore
\begin{align*}
\frac{1}{\rho_n^c}\abs{u(x)-u(y)}^2
&\leq\frac{1}{\rho_n^c}E_{n^{-}}^c[u]\abs{x-y}\leq\E[u]\abs{x-y},
\end{align*}
which leads to 
\[
\abs{u(x)-u(y)}\leq\left(\rho_n^c\right)^{1/2}\E[u]^{\frac{1}{2}}\abs{x-y}^{1/2}\leq\E^{1/2}[u]\abs{x-y}^{1/2}.
\]
The same calculations apply if $x\in e\in\J_n$ and $y\in W_n$ is one of its endpoints.
\item[(3)] If $x,y\in W_n$ are not neighbors we proceed as follows: Consider a chain of points $x_n,y_{n+1},x_{n+2},y_{n+2},\ldots,x_{n+k-1},y_{n+k}\in~V_\ast$ such that $x_{n+j}\stackrel{n+j+1}{\sim}y_{n+j+1}$ in $W_{n+j}$ and $(y_{n+j+1},x_{n+j+1})\in\J_{n+j}\setminus\J_{n+j-1}$ for each $0\leq j\leq k-1$ (see Figure~\ref{chain 1}). 
\begin{figure}[h!tpb]
\centering
\begin{tikzpicture}[scale=0.45]
\coordinate (p_1) at (0,0);
\fill (p_1) circle (2pt);
\coordinate (p_2) at (3,5.1961);
\fill (p_2) circle (2pt);
\coordinate[label=below: \footnotesize{$x_1$}] (p_3) at (6,0);
\fill (p_3) circle (3.5pt);
\coordinate (p_12) at (1,1.732);
\fill (p_12) circle (2pt);
\coordinate (p_21) at (2,3.4641);
\fill (p_21) circle (2pt);
\coordinate (p_13) at (2,0);
\fill (p_13) circle (2pt);
\coordinate (p_31) at (4,0);
\fill (p_31) circle (2pt);
\coordinate[label=right: \footnotesize{$x_2$}] (p_23) at (4,3.4641);
\fill (p_23) circle (3.5pt);
\coordinate[label=right: \footnotesize{$y_2$}] (p_32) at (5,1.732);
\fill (p_32) circle (3.5pt);
\coordinate (p_112) at (1/3,1.732/3);
\fill (p_112) circle (2pt);
\coordinate (p_121) at (2/3,3.4641/3);
\fill (p_121) circle (2pt);
\coordinate (p_113) at (2/3,0);
\fill (p_113) circle (2pt);
\coordinate (p_131) at (4/3,0);
\fill (p_131) circle (2pt);
\coordinate (p_123) at (4/3,3.4641/3);
\fill (p_123) circle (2pt);
\coordinate (p_132) at (5/3,1.732/3);
\fill (p_132) circle (2pt);
\coordinate (p_212) at (1/3+2,1.732/3+3.4641);
\fill (p_212) circle (2pt);
\coordinate (p_221) at (2/3+2,3.4641/3+3.4641);
\fill (p_221) circle (2pt);
\coordinate (p_213) at (2/3+2,3.4641);
\fill (p_213) circle (2pt);
\coordinate[label=below: \footnotesize{$y_3$}] (p_231) at (4/3+2,3.4641);
\fill (p_231) circle (3.5pt);
\coordinate (p_223) at (4/3+2,3.4641/3+3.4641);
\fill (p_223) circle (2pt);
\coordinate (p_232) at (5/3+2,1.732/3+3.4641);
\fill (p_232) circle (2pt);
\coordinate (p_312) at (1/3+4,1.732/3);
\fill (p_312) circle (2pt);
\coordinate (p_321) at (2/3+4,3.4641/3);
\fill (p_321) circle (2pt);
\coordinate (p_313) at (2/3+4,0);
\fill (p_313) circle (2pt);
\coordinate (p_331) at (4/3+4,0);
\fill (p_331) circle (2pt);
\coordinate (p_323) at (4/3+4,3.4641/3);
\fill (p_323) circle (2pt);
\coordinate (p_332) at (5/3+4,1.732/3);
\fill (p_332) circle (2pt);

\coordinate (A) at (9,0);
\coordinate (B) at (4,20.7844/3);

\draw (p_12) -- (p_21)  (p_23) -- (p_32)   (p_31) -- (p_13);
\draw (p_112) -- (p_121)  (p_132) -- (p_123)   (p_131) -- (p_113)
(p_212) -- (p_221)  (p_232) -- (p_223)   (p_231) -- (p_213)
(p_312) -- (p_321)  (p_332) -- (p_323)   (p_331) -- (p_313);
\draw (p_3) -- (A) (p_2) -- (B);
\end{tikzpicture}
\caption{\small{Chain with $x_1\in W_{1}$, $y_2,x_2\in W_{2}$ and $y_3\in W_{3}$.}}
\label{chain 1}
\end{figure}

If there exists some $k>1$ such that $x:=x_n\in W_n$ and $y:=y_{n+k}\in W_{n+k}\setminus W_{n+k-1}$, then, $\abs{x_{n+j}-y_{n+j+1}}=\left(\frac{1-\alpha}{2}\right)^{n+j+1}$ and $\abs{y_{n+j}-x_{n+j}}=\alpha\left(\frac{1-\alpha}{2}\right)^{n+j-1}$ and we get that
\begin{align*}
\abs{u(x)-u(y)}&\leq\sum_{j=0}^{k-1}\abs{u(x_{n+j})-u(y_{n+j+1})}+\sum_{j=1}^{k-1}\abs{u(y_{n+j})-u(x_{n+j})}\\
&\leq \E^{1/2}[u]\sum_{j=0}^{k-1} \abs{x_{n+j}-y_{n+j+1}}^{l_{\alpha}}+\E^{1/2}[u]\sum_{j=1}^{k}\abs{y_{n+j}-x_{n+j}}^{1/2}\\
&= \E^{1/2}[u]\left(\frac{1-\alpha}{2}\right)^{(n+1)l_{\alpha}}\sum_{j=0}^{k-1} \left(\frac{1-\alpha}{2}\right)^{l_{\alpha}j}\\
&+\E^{1/2}[u]\alpha^{1/2}\left(\frac{1-\alpha}{2}\right)^{\frac{n-1}{2}}\sum_{j=1}^{k}\left(\frac{1-\alpha}{2}\right)^{j/2}.
\end{align*}
Since $l_{\alpha}<1/2$, $\alpha\left(\frac{1-\alpha}{2}\right)^{-1}<1$ and $\left(\frac{1-\alpha}{2}\right)^{l_{\alpha}}<1$, we get that
\begin{align*}
\abs{u(x)-u(y)}&\leq \E^{1/2}[u]\left(\frac{1-\alpha}{2}\right)^{nl_{\alpha}}\sum_{j=0}^{k-1} \left(\frac{1-\alpha}{2}\right)^{l_{\alpha}j}\\
&+\E^{1/2}[u]\left(\frac{1-\alpha}{2}\right)^{nl_{\alpha}}\sum_{j=1}^{k}\left(\frac{1-\alpha}{2}\right)^{l_{\alpha}j}\\
&= 2\E^{1/2}[u]\left[1-\left(\frac{1-\alpha}{2}\right)^{l_{\alpha}}\right]^{-1}\left(\frac{1-\alpha}{2}\right)^{nl_{\alpha}}.
\end{align*}
Finally, $\left(\frac{1-\alpha}{2}\right)^{n}\leq\abs{x-y}$ because $y\notin W_n$ by assumption, hence, if we set $C:=2\left[1-\left(\frac{1-\alpha}{2}\right)^{l_{\alpha}}\right]^{-1}$, we obtain
\begin{equation*}
\abs{u(x)-u(y)}\leq C \E^{1/2}[u]\abs{x-y}^{l_{\alpha}}.
\end{equation*}

In the case $k=0$, i.e. $x,y\in W_n\setminus W_{n-1}$ are not $n$-neighbors, we can join them by at most two such chains, say $x:=x_n,\ldots, y_{n+k}$ and $y:=x'_n,\ldots, y'_{n+k}$ for some $k\in\N$ and an extra segment $(y_{n+k},y'_{n+k})$ of length $\alpha\left(\frac{1-\alpha}{2}\right)^{n+k-1}$ (in the case that  $y_{n+k}\neq y'_{n+k}$).
The triangular inequality and last calculation leads to 
\begin{align*}
\abs{u(x)-u(y)} 
&\leq 2C\E^{1/2}[u]\left(\frac{1-\alpha}{2}\right)^{(n+1)l_{\alpha}}+\E^{1/2}[u]\alpha^{1/2}\left(\frac{1-\alpha}{2}\right)^{\frac{n+k-1}{2}}
\end{align*}
and by using again the fact that $l_{\alpha}<1/2$, $\alpha\left(\frac{1-\alpha}{2}\right)^{-1}<1$, $k\geq 1$ and $\abs{x-y}>\left(\frac{1-\alpha}{2}\right)^{(n+1)}$, we obtain
\begin{align*}
\abs{u(x)-u(y)}&\leq (2C+1)\E^{1/2}[u]\left(\frac{1-\alpha}{2}\right)^{(n+1)l_{\alpha}}\nonumber\\
&\leq (2C+1)\E^{1/2}[u]\abs{x-y}^{l_{\alpha}}.
\end{align*}

\item[(4)] If $x,y\in J_n\setminus J_{n-1}$ do not belong to the same line segment, then there exists $e_1,e_2\in\J_n$ such that $x\in e_1$, $y\in e_2$. Now we can join both points as follows: consider $x'\in W_n$ the nearest endpoint of $e_1$ to $x$, and $y'\in W_n$ the nearest in $e_2$ to $y$. Then, by an analogous calculation as the previous case and the applying the triangular inequality we have
\begin{align*}
\abs{u(x)-u(y)}
&\leq (2C+3)\E^{1/2}[u]\abs{x-y}^{l_{\alpha}}.
\end{align*}

\enlargethispage{0.5cm}
\begin{figure}[h!tpb]
\centering
\begin{tikzpicture}[scale=0.45]
\coordinate (p_1) at (0,0);
\fill (p_1) circle (2pt);
\coordinate (p_2) at (3,5.1961);
\fill (p_2) circle (2pt);
\coordinate (p_3) at (6,0);
\fill (p_3) circle (2pt);
\coordinate (p_12) at (1,1.732);
\fill (p_12) circle (2pt);
\coordinate (p_21) at (2,3.4641);
\fill (p_21) circle (2pt);
\coordinate (p_13) at (2,0);
\fill (p_13) circle (2pt);
\coordinate (p_31) at (4,0);
\fill (p_31) circle (2pt);
\coordinate (p_23) at (4,3.4641);
\fill (p_23) circle (3pt);
\coordinate[label=below: \footnotesize{$x'$}] (p_32) at (5,1.732);
\fill (p_32) circle (3pt);
\coordinate (p_112) at (1/3,1.732/3);
\fill (p_112) circle (2pt);
\coordinate (p_121) at (2/3,3.4641/3);
\fill (p_121) circle (2pt);
\coordinate (p_113) at (2/3,0);
\fill (p_113) circle (2pt);
\coordinate (p_131) at (4/3,0);
\fill (p_131) circle (2pt);
\coordinate (p_123) at (4/3,3.4641/3);
\fill (p_123) circle (2pt);
\coordinate (p_132) at (5/3,1.732/3);
\fill (p_132) circle (2pt);
\coordinate (p_212) at (1/3+2,1.732/3+3.4641);
\fill (p_212) circle (2pt);
\coordinate (p_221) at (2/3+2,3.4641/3+3.4641);
\fill (p_221) circle (2pt);
\coordinate (p_213) at (2/3+2,3.4641);
\fill (p_213) circle (2pt);
\coordinate[label=below: \footnotesize{$y'$}] (p_231) at (4/3+2,3.4641);
\fill (p_231) circle (3pt);
\coordinate (p_223) at (4/3+2,3.4641/3+3.4641);
\fill (p_223) circle (2pt);
\coordinate (p_232) at (5/3+2,1.732/3+3.4641);
\fill (p_232) circle (2pt);
\coordinate (p_312) at (1/3+4,1.732/3);
\fill (p_312) circle (2pt);
\coordinate (p_321) at (2/3+4,3.4641/3);
\fill (p_321) circle (2pt);
\coordinate (p_313) at (2/3+4,0);
\fill (p_313) circle (2pt);
\coordinate (p_331) at (4/3+4,0);
\fill (p_331) circle (2pt);
\coordinate (p_323) at (4/3+4,3.4641/3);
\fill (p_323) circle (3pt);
\coordinate (p_332) at (5/3+4,1.732/3);
\fill (p_332) circle (2pt);

\coordinate (A) at (9,0);
\coordinate[label=right: \footnotesize{$x$}] (x) at (4.5/3+4,2.598/3);
\fill (x) circle (3pt);
\coordinate (B) at (4,20.7844/3);
\coordinate[label=above: \footnotesize{$y$}] (y) at (4/3+1.6,3.4641);
\fill (y) circle (3pt);
\draw (p_12) -- (p_21)  (p_23) -- (p_32)   (p_31) -- (p_13);
\draw (p_112) -- (p_121)  (p_132) -- (p_123)   (p_131) -- (p_113)
(p_212) -- (p_221)  (p_232) -- (p_223)   (p_231) -- (p_213)
(p_312) -- (p_321)  (p_332) -- (p_323)   (p_331) -- (p_313);
\draw (p_3) -- (A) (p_2) -- (B);
\end{tikzpicture}
\caption{\small{Chain with $x\in e_1$, $y\in e_{2}$.}}
\label{chain 2}
\end{figure}
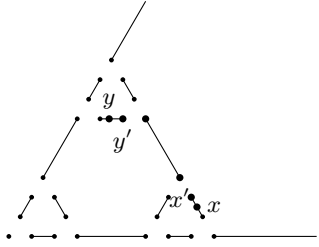

Now, choosing $\tilde{C}:=2C+3$, it follows from cases (3) and (4) that 
\[\abs{u(x)-u(y)}\leq \tilde{C}\, \E^{1/2}[u]\abs{x-y}^{l_{\alpha}}\]
for all $x,y\in W_s$, hence $u$ is uniformly H\"older-continuous.

\item[(5)] The case when $x\in J_n$ and $y\in W_n$ follows by combining the two last cases.
\end{enumerate}
\end{proof}

This result allows us to prove the next proposition, which has important consequences.

\begin{prop}\label{prop RcompE}
The resistance metric $R$ associated with $(\E,\F)$ and the Euclidean metric induce the same topology on $\Ka$.
\end{prop}
\begin{proof}
We follow the standard proof in~\cite[Proposition 7.18]{Bar95}. On the one hand, given a sequence $(x_n)_{n\in\N}\subseteq\Ka$ that converges to $x\in\Ka$ with respect to the Euclidean metric, it follows from Lemma~\ref{lem Dast in cont} that
\[
R(x_n,x)\leq\tilde{C}\abs{x_n-x}^{2l_{\alpha}}\xrightarrow{n\to\infty}0
\]
and hence $(x_n)_{n\in\N}$ converges with respect to the resistance metric too.

On the other hand, let $(x_n)_{n\in\N}\subseteq\Ka$ converge to some $x\in\Ka$ with respect to the resistance metric. Then, $\forall\,\varepsilon>0$, $\exists\,N\in\N$ such that $R(x_n,x)<\varepsilon$ for all $n\geq N$. 

Now, for each $\varepsilon>0$ we can construct a function $u\in\F$ such that $u(x)=1$ and $\supp(u)\subseteq B_{\varepsilon}(x)$ as follows:
without loss of generality, suppose that $x,y\in V_k$ for some $k\in\N_0$. Now consider $B_{\varepsilon}(x)\subset (a_e,b_e)$ for some $e\in\J_k$ such that $y\in V_n\setminus B_{\varepsilon}(x)(=:V_{\varepsilon})$ and define $v\colon V_n\to\R$ to be some smooth function with $v(x):=1$ and $v_{\vert_{V_{\varepsilon}}}\equiv 0$. Then, $v\in\D_n$ because $\E_n[v]<\infty$. By defining $u\colon\Ka\to\R$ as the harmonic extension of $v$ we get that $\E[u]=\E_n[v]<\infty$ and thus $u\in\F$, $u(x)=1$ and $u(y)=0$ as desired. For this function it holds that
\begin{equation*}
R(x,y)>\frac{1}{\E[u]}>0\qquad\forall\,y\in\Ka\setminus B_{\varepsilon}(x),
\end{equation*}
hence there exists $N\in\N_0$ such that $x_n\in B_{\varepsilon}(x)$ for all $n\geq N$, which means that $(x_n)_{n\in\N}$ converges with respect to the Euclidean norm too. This finishes the proof.
\end{proof}

\begin{thm}
The pair $(\E,\F)$ given by
\begin{align*}
\F&:=\{u\colon\Ka\to\R~\vert~u\in\D_\ast,\;\lim_{n\to\infty}\E_n[u_{\vert_{V_n}}]<\infty\}\\
\E[u]&:=\lim_{n\to\infty}\E_n[u_{\vert_{V_n}}]
\end{align*}
is a resistance form on $\Ka$.
\end{thm}

\begin{proof}
First note that any interval $(a_e,b_e)$, $e\in\J_n$ contains a countable dense set $D^e$ that can be approximated by finite sets $D^e_n$ (think of approximating rational points on the interval). For each $n\geq 1$, we define the finite sets 
\[
\widetilde{V}_n:=W_n\cup \bigcup_{e\in\J_n}D^e_n.
\]
Our first step in the proof is the construction of a compatible sequence of resistance forms $(\widetilde{\E}_n,\ell(\widetilde{V}_n))$, where 
\begin{equation*}
\widetilde{\E}_n[u]:=\inf\{\E_n[v]~\vert~v\in\D_n\text{ and }v_{\vert_{\widetilde{V}_n}}\equiv u\}.
\end{equation*}
In order to show that $(\widetilde{\E}_n,\ell(\widetilde{V}_n))$ is a resistance form, we follow the lines of~\cite{BCF+07}. Using the proof of Proposition~\ref{prop renorm}, we have that
\begin{equation*}
\frac{5}{3}r_n^d+r_n^c=1\qquad\forall\,n\geq 1,
\end{equation*}
which implies the equality
\begin{equation*}
\frac{5}{3}\rho_n^d+\rho_n^c=\prod_{i=1}^{n-1}r_i^d\left(\frac{5}{3}r_n^d+r_n^c\right)=\prod_{i=1}^{n-1}r_i^d=\rho_{n-1}^d
\end{equation*}
and hence $(\widetilde{\E}_n,\ell(\widetilde{V}_n))$ is a resistance form.

On the other hand, we have that for any $n\geq 1$ and $u\in\ell(\widetilde{V_n})$
\begin{align*}
\widetilde{\E}_n[u]&=\inf\{\E_n[v]~\vert~v\in\D_n\text{ and }v_{\vert_{\widetilde{V}_n}}\equiv u\}\\
&=\inf\{\E_{n+1}[\tilde{v}]~\vert~\tilde{v}\in\D_{n+1}\text{ and }\tilde{v}_{\vert_{V_n}}\equiv v\}\\
&=\inf\{\E_{n+1}[\tilde{v}]~\vert~\tilde{v}\in\D_{n+1}\text{ and }\tilde{v}_{\vert_{V_n}}\equiv u\}\\
&=\inf\{\widetilde{\E}_{n+1}[\widetilde{u}]~\vert~\widetilde{u}\in\ell(\widetilde{V}_n)\text{ and }\widetilde{u}_{\vert_{\widetilde{V}_n}}\equiv u\},
\end{align*}
and therefore a compatible sequence of resistance forms.

Finally, if we define 
\begin{align*}
\widetilde{\F}&:=\{u\in\ell(\widetilde{V}_\ast)~\vert~\lim\limits_{n\to\infty}\widetilde{\E}_n[u_{\vert_{\widetilde{V}_n}}]<\infty\},\\
\widetilde{\D}&:=\{u\in\D_\ast~\vert~\lim_{n\to\infty}\E_n[u_{\vert_{V_n}}]<\infty\},
\end{align*}
applying Lemma~\ref{lem Dast in cont} and Proposition~\ref{prop RcompE}, it follows from~\cite[Theorem 3.13]{Kig12} that
\begin{align*}
\F&=\{u\colon\Ka\to\R~\vert~u_{\vert_{V_\ast}}\in\widetilde{\D}\}=\{u\colon \Ka\to\R~\vert~u_{\vert_{\widetilde{V}_n}}\in\widetilde{\F}\}\\
\E[u]&=\lim_{n\to\infty}\E_n[u_{\vert_{V_n}}]=\lim_{n\to\infty}\widetilde{\E}_n[u_{\vert_{\widetilde{V}_n}}]
\end{align*} 
is a resistance form on $\Ka$. Moreover, $\F\subset C(\Ka)$.
\end{proof}

\begin{cor}\label{cor: EKa is regular}
The resistance form $(\E,\F)$ is regular.
\end{cor}
\begin{proof}
In view of Proposition~\ref{prop RcompE}, $\Ka$ is $R$-compact, hence $(\E,\F)$ is regular by~\cite[Corollary 6.4]{Kig12}.
\end{proof}

We finish this paragraph with a scaling result for $(\E,\D)$.

\begin{lemma}
Let $u_i:=u\circ G_i$ for any $u\in\F$. Then
\begin{align*}
\E[u]&=\sum_{i=1}^3\left(\frac{5}{3}\E^d[u_i]+\frac{5}{3}\left(\frac{1-\alpha}{2}\right)^{2}\E^c[u_i]+\left(\frac{1-\alpha}{2}\right)\widetilde{\E}^c[u_i]\right)+\E_1^c[u],
\end{align*}
where 
\[
\E^d[u]:=\lim_{n\to\infty}\frac{1}{\rho_n^d}E_n^d[u_{\vert_{V_n}}],\qquad
\E^c[u]:=\lim_{n\to\infty}\sum_{k=1}^n\frac{1}{\rho_k^c}E^c_{k^{-}}[u_{\vert_{V_n}}],
\]
and
\[
\widetilde{\E}^c[u]:=\lim_{n\to\infty}\sum_{k=1}^n\frac{1}{\rho_n^d}E^c_{k^{-}}[u_{\vert_{V_n}}].
\]
\end{lemma}
\begin{proof}
First note that $\rho_n^c<\rho_n^d$ and thus $\tilde{\E}^c$ is finite for any $u\in\F$.

On the one hand,
\[
\E^d_{n+1}[u]=\frac{\rho^d_n}{\rho^d_{n+1}}\sum_{i=1}^3\E_n^d[u_i]=\frac{5+3d_{n+1}}{3}\sum_{i=1}^3\E_n^d[u_i].
\]
Letting $n\to\infty$ in both sides of the equality we get
\begin{equation}\label{eq lemma scaling}
\E^d[u]=\frac{5}{3}\sum_{i=1}^3\E[u_i].
\end{equation}
On the other hand
\begin{align*}
\E_{n+1}^c[u]&=\sum_{k=1}^{n+1}\frac{1}{\rho_k^c}\int_{J_k\setminus J_{k-1}}\abs{\nabla u}^2dx\\
&=\frac{1}{\rho_1^c}\int_{J_1}\abs{\nabla u}^2dx+\sum_{i=1}^3\sum_{k=1}^n\frac{1}{\rho^c_{k+1}}\int_{G_i(J_k\setminus J_{k-1})}\abs{\nabla u}^2dx\\
&=\E_1^c[u]+\sum_{i=1}^3\sum_{k=1}^n\frac{1}{\rho_{k+1}}\left(\frac{2}{1-\alpha}\right)\int_{J_k\setminus J_{k-1}}\abs{\nabla u_i}^2dx\\
&=\E_1^c[u]+\left(\frac{2}{1-\alpha}\right)\sum_{i=1}^3\sum_{k=1}^n\frac{\rho_k^c}{\rho_{k+1}^c}\frac{1}{\rho_k^c}E^c_{k^{-}}[u_i]\\
&=\E_1^c[u]+\left(\frac{2}{1-\alpha}\right)\sum_{i=1}^3\sum_{k=1}^n\frac{5+3d_{k+1}}{3}\frac{2}{1-\alpha}\frac{1}{\rho_k^c}E^c_{k^{-}}[u_i].
\end{align*}
Splitting $\frac{5+d_{k+1}}{3}$ into its two summands and since $\frac{d_{k}}{\rho_k^c}=\frac{1}{\rho_k^d}$ we get that
\begin{align*}
\E_{n+1}^c[u]&=\E_1^c[u]+\sum_{i=1}^3\left(\frac{5}{3}\left(\frac{2}{1-\alpha}\right)^2\E_n^c[u_i]+\left(\frac{2}{1-\alpha}\right)\sum_{k=1}^n\frac{1}{\rho_k^d}E^c_{k^{-}}[u_i]\right).
\end{align*}
Letting $n\to\infty$ in both sides of the equality we get
\begin{equation*}
\E^c[u]=\sum_{i=1}^3\left(\frac{5}{3}\left(\frac{1-\alpha}{2}\right)^{2}\E^c[u_i]+\frac{1-\alpha}{2}\widetilde{\E}^c[u_i]\right)+\frac{1}{\rho_1^c}\E_1^c[u]
\end{equation*}
which together with~\eqref{eq lemma scaling} proves the assertion.
\end{proof}

By iterating the calculations in the previous proof we get the following scaling for an arbitrary level $m$:
\begin{cor}\label{cor scaling prop EKa}
For each $m\in\N$, $w\in\A^m$ and $u\in\F$ it holds that
\begin{align*}
\E[u]&=\left(\frac{5}{3}\right)^m\sum_{w\in\Am}\left(\E^d[u_w]+\left(\frac{1-\alpha}{2}\right)^{2m}\E^c[u_w]\right)\\
&+\left(\frac{1-\alpha}{2}\right)^m\sum_{w\in\Am}\widetilde{\E}_m^c[u_w]+\sum_{k=1}^{m-1}\sum_{w\in\A^{k-1}}\left(\frac{2}{1-\alpha}\right)^k\frac{1}{\rho_k^c}E_1^c[u_w],
\end{align*}
where $u_w:=u\circ G_w$, and 
\[
\widetilde{\E}_m^c[u]:=\lim_{n\to\infty}\sum_{k=1}^n P_{k,m}\left(\frac{1}{3},5,\left(\frac{1-\alpha}{2}\right)^k\right))E^c_{k^{-}}[u]
\]
for some polynomial $P_{k,m}$ of degree $m$.
\end{cor}
The same results hold for the bilinear from $\E(u,v)$, $u,v\in\F$.
\subsubsection{Dirichlet form}
In order to obtain a Dirichlet form from the resistance form, we need a locally finite regular measure $\mu$ on $\Ka$. Due to the non self-similarity of $\Ka$, there is no ``canonical'' choice of such measure. Hence we will not specify it until the next section, when it becomes necessary for the study of the associated Laplacian.

\medskip

Let $\mu$ be an arbitrary finite Radon measure on $\Ka$ and let $L^2(\Ka,\mu)$ be the associated Hilbert space. From Lemma~\ref{lem Dast in cont} it follows that $\F\subseteq L^2(\Ka,\mu)$ so we can define
\begin{equation}\label{eq: Def EKa1}
\E_{(1)}(u,v):=\E(u,v)+\int_{\Ka}u\cdot v\,d\mu\qquad u,v\in\F.
\end{equation}

By~\cite[Theorem 2.4.1]{Kigami01} this is an inner product in $\F$ and thus we can consider the norm $\norm{\cdot}_{\E_{(1)}}:=\E_{(1)}^{1/2}$.

\medskip

Let $C_0(\Ka)$ denote the set of compactly supported continuous functions in $\Ka$ ( in fact $C(\Ka)$) and let $\D$ be the closure of $C_0(\Ka)\cap\F$ with respect to the norm $\norm{\cdot}_{\E_{(1)}}$. On the one hand, it follows from Corollary~\ref{cor: EKa is regular} that $\D$ is dense in $C(\Ka)$. On the other hand, it is a well known result from classical analysis that $C_0(\Ka)$ is dense in $L^2(\Ka,\mu)$. Thus $\D$ is dense in $L^2(\Ka,\mu)$ too and the pair $(\E,\D)$ is called the Dirichlet form \textit{induced by the resistance form} $(\E, \F)$. Moreover, since $\Ka$ is $R$-compact by Proposition~\ref{prop RcompE}, $\D=\F$.

\medskip

\begin{thm}\label{thm: EKa loc reg DF}
The Dirichlet form $(\E,\D)$ on $L^2(\Ka,\mu)$ is local and regular.
\end{thm}
\begin{proof}
By Corollary~\ref{cor: EKa is regular} $(\E,\F)$ is a regular resistance form, hence by~\cite[Theorem 9.4]{Kig12} its associated Dirichlet form $(\E,\D)$ is a regular Dirichlet form.

If we consider $u,v\in\D$ such that $\supp(u)\cap\supp(v)=\emptyset$, since $\supp(u)$ and $\supp(v)$ are compact sets, there exists some $n\in\N$ such that for all $w\in\An$, either $\supp(u)\cap G_{w}(\Ka)=\emptyset$ or $\supp(v)\cap G_{w}(\Ka)=\emptyset$. By Corollary~\ref{cor scaling prop EKa} we get that $\E(u,v)=0$, hence the form is local.
\end{proof}

\section{Measure and Laplacian on $\Ka$}
Since the properties of the Laplacian associated with the Dirichlet form $(\E,\D)$  strongly depend on the choice of the measure on $\Ka$, we need to fix one up to this point. The one constructed here has been chosen in this particular manner for technical reasons. We recover at this stage the parameter $\alpha$ in our discussion and write $(\EKa,\Da)$ for the Dirichlet form $(\E,\D)$, as well as all other dependencies.
\subsection{Measure on $\Ka$}
The following result gives a decomposition of $\Ka$ that will be very useful in the definition of the measure $\mu_{\alpha,\beta}$.
\begin{lemma}\label{lem: Ka=Fa cup Ja}
Let $\Fa$ be the unique nonempty compact subset of $\R^2$ satisfying $\Fa=\bigcup\limits_{i=1}^3\Gai(\Fa)$ and define $\Ja:=\bigcup\limits_{n\in\N_0}\Jan$. Then,
\[\Ka=\Fa\,\dot{\cup}\,\Ja.\]
\end{lemma} 
\begin{proof}
See~\cite[Lemma 2.1.1]{ARThesis13}
\end{proof}

Now, let $\lambda$ denote the $1-$dimensional Hausdorff measure and consider $\beta$ any positive number satisfying
\begin{equation}\label{eq: property of beta}
0<\beta<\left(\frac{2}{3(1-\alpha)}\right)^2.
\end{equation} 

On the one hand, we define the self-similar measure on $\Fa$ given by
\[\mu_{\alpha}^d(A):=\frac{1}{2\mathcal{H}^{\delta_\alpha}(F_{\alpha})}\mathcal{H}^{\delta_\alpha}_{\vert_{F_{\alpha}}}(A)\qquad\text{for } A\subseteq\R^2\text{  Borel},\]
where $\delta_\alpha:=\dim_H\Ka=\frac{\ln 3}{\ln 2-\ln (1-\alpha)}$ and $\mathcal{H}^{\delta_\alpha}(\cdot)$ denotes the $\delta_\alpha$-dimensional Hausdorff measure. 

\medskip

On the other hand, we define the Radon measure on $\R^2$ given by
\[
\mu_{\alpha,\beta}^c(A):=\frac{1}{2\tilde{\mu}_{\alpha,\beta}^c(\Ja)}\tilde{\mu}_{\alpha,\beta}^c(A)\qquad \text{for }A\subseteq\R^2\text{  Borel},
\]
where
\[\tilde{\mu}_{\alpha,\beta}^c(A):=\sum_{e\in\Jae}\beta_e\lambda(A\cap e ),\]
and $\beta_e:=\beta^k$ if $e\in\Jake{k+1}\setminus\Jake{k}$, $\beta$ being the constant chosen in~\eqref{eq: property of beta}. Note that $\supp (\muac)=\Ja$ and the choice of $\beta$ ensures that $\tilde{\mu}_{\alpha,\beta}^c(\Ja)<\infty$.

\medskip

In view of Lemma~\ref{lem: Ka=Fa cup Ja} we may define a finite Radon measure on $\R^2$ as the sum
\[\mu_{\alpha,\beta}(A):=\muad(A\cap\Fa)+\mu_{\alpha,\beta}^c(A\cap\Ja) \qquad\text{for } A\subseteq\R^2\text{  Borel}.\]
Note that $\supp (\mua)=\Ka$ and $\mua(\Ka)=1$.


\subsection{Laplacian}
It is known from the theory of Dirichlet forms that $(\EKa,\Da)$ defines a Laplacian on $\Ka$ in the weak sense as the unique non-positive, self-adjoint and densely defined operator $\Delta_{\mua}\colon\D(\La)\to L^2(\Ka,\mua)$ such that for any $u\in\D(\La)$ 
\[\EKa(u,v)=(-\La u,v)_{\mua}\qquad\forall\;v\in\Da.\]

\medskip

Following the same arguments as in~\cite{Kig03} it can be proven that all functions in the domain $\D(\La)$ have normal derivative in the sense of~\cite[Definition 2.3.1]{Strichartz06} equal zero on the boundary $V_{\alpha,0}$. Hence we can say that the functions in $\D(\La)$ satisfy homogeneous Neumann boundary conditions, and from now on we adopt the notation $\LaN$ for the Neumann Laplacian associated to $(\EKa,\Da)$.

\medskip

The Laplacian $\LaD$ subject to Dirichlet boundary conditions is defined ana\-lo\-gous\-ly by modifying the domain $\Da$ to $\Dao:=\{u\in\Da\,\vert\;u_{\vert_{\Va{0}}}\equiv 0\}$.

\begin{thm}\label{thm: Spec of Lap}
The operator $-\LaN$ has pure point spectrum consisting of countable many non-negative eigenvalues with finite multiplicity and only accumulation point at $+\infty$. The same holds for the operator $-\LaD$.
\end{thm}

\begin{proof}
Proposition~\ref{prop RcompE} together with~\cite[Theorem 10.4]{Kig12} imply that the Dirichlet form $(\EKa,\Da)$ has a jointly continuous kernel. Hence the semigroup $e^{-\LaN t}$ is ultracontractive by~\cite[Lemma 2.1.2]{Davies90} and the claim follows from~\cite[Theorem 2.1.4]{Davies90}. 
\end{proof}

\section{Spectral dimension}
This last result on the spectrum of the operators $-\LaN$ and $-\LaD$ allows us to study the asymptotic behavior of the eigenvalue counting function associated with each of them. In the following, whenever an statement holds for both operators, we will use the notation $\Delta_{\mua}^{N/D}$.

\medskip

\begin{defn}\label{defn: def evcf}
The \textit{eigenvalue counting function} of $-\Delta_{\mua}^{N/D}$ is defined for each $x\geq 0$ as
\[N_{N/D}(x):=\#\{\kappa~\vert~ \kappa\text{ eigenvalue of }-\Delta_{\mua}^{N/D}\text{ and }\kappa\leq x\},\]
counted with multiplicity.
\end{defn}

\medskip

\begin{rem} Given a Dirichlet form $(E,D)$ on a Hilbert space $H$, we say that $\kappa\in\R$ is an \textit{eigenvalue of} $E$ if and only if
 there exists $u\in D$, $u\neq 0$, such that $E(u,v)=\kappa(u,v)$ for all $v\in D$. 
 
The eigenvalue counting function can thus be defined for a Dirichlet form $( E, D)$ on a Hilbert space $H$ at any $x>0$ by
\[N(x;E,D):=\#\{\kappa~\vert~ \kappa\text{ eigenvalue of } E\text{ and } \kappa\leq x\},\]
counted with multiplicity.

Furthermore, it follows from~\cite[Proposition 4.2]{Lap91} that $N_N(x)=N(x;\EKa,\Da)$ and $N_D(x)=N(x;\EKa^0,\Da^0)$.
\end{rem}

\medskip

Given two functions $f,g\colon\R\to\R$, let us write $f(x)\asymp g(x)$ if there exist constants $C_1,C_2>0$ such that $C_1f(x)\leq g(x)\leq C_2f(x)$.

\medskip

The \textit{spectral dimension} of $\Ka$ describes the asymptotic behavior of both eigenvalue counting functions and it is defined as the number  $d_S(\Ka)>0$ (in case it exists) such that
\begin{equation}\label{eq: def d_S  Ka}
N_{N/D}(x)\asymp x^{\frac{d_S}{2}}\qquad\text{ as }x\to\infty.
\end{equation}

\medskip

In this section we estimate the eigenvalue counting function of the Laplacian $-\La$:
 
\begin{thm}\label{prop estimation evcf}
There exist constants $C_{\alpha,1},C_{\alpha,\beta,1},C_{\alpha,2},C_{\alpha,\beta,2}>0$ and $x_0>~0$ such that
\begin{equation}\label{eq prop esti evcf}
C_{\alpha,1}x^{\frac{\log 3}{\log 5}}+C_{\alpha,\beta,1}x^{\frac{1}{2}}\leq N_D(x)\\
\leq N_N(x)\leq C_{\alpha,2}x^{\frac{\log 3}{\log 5}}+C_{\alpha,\beta,2} x^{\frac{1}{2}}
\end{equation}
for all $x\geq x_0$. 
\end{thm}

\medskip

This leads to
\begin{cor}\label{thm d_S Ka}
For any $0<\alpha<1/3$,
\begin{equation*}
d_S (\Ka)=\frac{2\log 3}{\log 5}.
\end{equation*}
\end{cor}

The proof of this result is based on the \textit{minimax principle} for the eigenvalues of non-negative self-adjoint operators and it follows ideas of~\cite{Kaj10}. Details about the minimax principle can be found in~\cite[Chapter 4]{Davies95}. 
\subsection{Spectral asymptotics of the Laplacian}
This section is devoted to the proof of Theorem~\ref{prop estimation evcf} and is divided into two parts: upper and lower bound of the eigenvalue counting function.

\subsubsection*{Upper bound}

We start by decomposing $\Ka$ into suitable pieces where we have a better control of the eigenvalues. For each $m\geq 0$ and any word $w\in\A^m$ we write $\Kak{w}:=\Ga{w}(\Ka)$ and $\Kak{m}:=\bigcup\limits_{w\in\A^m}\Kak{w}$. Note that $\Ka=\Kak{m}\cup \Jak{m}$.

\medskip

On one hand, following the same construction as in Section 2, we approximate $\Kak{w}$ by the sets $V_{\alpha,w,n}$ as $\Ka$ was approximated by $\Van$. Then we can define a resistance form $(\E_{\Kak{w}},\F_{\Kak{w}})$ on $\Kak{w}$ by
\begin{align*}
\E_{\Kak{w}}(u,v)&:=\lim_{n\to\infty}\E_{\Kak{w}}^{(n)}(u_{\vert_{V_{\alpha,w,n}}},v_{\vert_{V_{\alpha,w,n}}}),\\
\F_{\Kak{w}}&:=\{u\colon\Kak{w}\to\R~\vert~\lim_{n\to\infty}\E_{\Kak{w}}^{(n)}(u_{\vert_{V_{\alpha,w,n}}},u_{\vert_{V_{\alpha,w,n}}})<\infty\}.
\end{align*}
Further, we consider the Dirichlet form $(\E_{\Kak{w}},\D_{\Kak{w}})$ on $L^2(\Kak{w},\mua{}_{\vert_{\Kak{w}}})$ induced by this resistance form.

\medskip

Finally we consider the Dirichlet form $(\E_{\Kak{m}},\D_{\Kak{m}})$ on $L^2(\Kak{m},\mua{}_{\vert_{\Kak{m}}})$ given by
\begin{equation}\label{def E_Kam}
\E_{\Kak{m}}:=\bigoplus_{w\in\A^m}\E_{\Kak{w}},\quad
\D_{\Kak{m}}:=\bigoplus_{w\in\A^m}\D_{\Kak{w}}.
\end{equation}

On the other hand, we consider the Dirichlet form $(\E_{\Jak{m}},\D_{\Jak{m}})$ given by
\begin{equation}\label{def E_Jam}
\E_{\Jak{m}}(u,v):=\sum_{k=1}^m\sum_{e\in\J_{\alpha,k}\setminus\J_{\alpha,k-1}}\frac{1}{\rho_k^c}\int_e u'v'\,dx,\quad \D_{\Jak{m}}:=\bigoplus_{e\in\J_{\alpha,m}} H^1(e,dx).
\end{equation}
For ease of the reading, we will drop the measure from the $L^2-$spaces whenever we consider $\mua$ or its restriction.

\begin{rem}
Note that for any $m\in\N$ and $w\in\A^m$, the domain $\F_{\Kak{w}}$ could also be characterized by $\F_{\Kak{w}}=\F_{\Ka}\circ\Ga{w}^{-1}$. In contrast to the self-similar case, the explicit expression of $\E_{\Kak{w}}[u]$ does \textit{not} coincide with $\E_{\Ka}[u\circ \Ga{w}^{-1}]$. However, by construction we have the useful identity
\[
\E_{\Ka}[u]=\sum_{w\in\A^m}\E_{\Kak{w}}[u_{\vert_{\Kak{w}}}]+\E_{\Jak{m}}[u_{\vert_{\Jak{m}}}],\quad u\in\F_{\Ka}.
\]

\end{rem}

\begin{lemma}\label{lem: Neumann upper +decomp}
For any $m\in\N_0$, let $H_{\Jak{m}}$ be the non-negative self-adjoint operator on $L^2(\Jak{m})$ associated with the Dirichlet form $(\E_{\Jak{m}},\D_{\Jak{m}})$ and let $H_{\Kak{m}}$ be the non-negative self-adjoint operator on $L^2(\Kak{m})$ associated with the Dirichlet form $(\E_{\Kak{m}},\D_{\Kak{m}})$. Then, $H_{\Jak{m}}$ and $H_{\Kak{m}}$ have both compact resolvent and for $N_{\Jak{m}}(x):=N(x;\E_{\Jak{m}},\D_{\Jak{m}})$ 
and $N_{\Kak{m}}(x):=N(x;\E_{\Kak{m}},\D_{\Kak{m}})$, 
it holds that 
\[N_N(x)\leq N_{\Kak{m}}(x)+N_{\Jak{m}}(x)\]
for any $x\geq 0$.
\end{lemma}
\begin{proof}
The statements about compactness of the resolvent are proved in Lemma~\ref{lem: Jam upper} and Lemma~\ref{lem: Neumann eigenv Kam} respectively.
 
First, $L^2(\Ka,\mua)=L^2(\Kak{m},\mua{}_{\vert_{\Kak{m}}})\oplus L^2(\Jak{m},\mua{}_{\vert_{\Jak{m}}})$ holds because $\Kak{m}\cap\Jak{m}=\emptyset$. By definition of $(\E_{\Kak{m}},\D_{\Kak{m}})$ and $(\E_{\Jak{m}},\D_{\Jak{m}})$ we have that $\EKa=\E_{\Kak{m}}\oplus\E_{\Jak{m}}$ and finally, since $\Da\subseteq\D_{\Kak{m}}\oplus~\D_{\Jak{m}}$, we get from~\cite[Proposition 4.2, Lemma 4.2]{Lap91} that
\begin{equation*}
N_N(x)\leq N(x;\EKa,\D_{\Kak{m}}\oplus\D_{\Jak{m}})=N_{\Kak{m}}(x)+N_{\Jak{m}}(x).
\end{equation*}
\end{proof}

\begin{lemma}\label{lem: Jam upper}
For each $m\in\N$, the non-negative self-adjoint operator $H_{\Jak{m}}$ associated with the Dirichlet form $(\E_{\Jak{m}},\D_{\Jak{m}})$ on $L^2(\Jak{m})$ has compact resolvent. Further, there exist a constant $C_{\alpha,\beta,2}>0$ depending on $\alpha$ and $\beta$, and $x_0>0$ such that
\begin{equation}\label{eq: Jam upper}
N_{\Jak{m}}(x)\leq C_{\alpha,\beta,2}\,x^{1/2}
\end{equation}
for all $x\geq x_0$, where $N_{\Jak{m}}(x):=N(x;\E_{\Jak{m}},\D_{\Jak{m}})$.
\end{lemma}
\begin{proof}
Note that the operator $H_{\Jak{m}}$ is the sum of classical one-dimensional Laplacians $-\Delta$ time constant restricted to the finite union of intervals $(a_e,b_e)$. Hence it has compact resolvent.

Let us now prove the inequality~\eqref{eq: Jam upper}.  
For any $e\in\Jake{m}\setminus\Jake{m-1}$ and $h\in H^1(e,dx)$ define
\[\tilde{h}(x):=\left\{\begin{array}{rl}
								h(x),&\text{if }x\in e,\\
								0,&\text{if }x\in\Jak{m}\setminus e.
\end{array}\right. \]
Then, $\tilde{h}\in\D_{\Jak{m}}$ and given $u\in\D_{\Jak{m}}$ an eigenfunction of $(\E_{\Jak{m}},\D_{\Jak{m}})$ with eigenvalue $\kappa$ we get
\begin{align*}
\int_e\nabla u\nabla h\,dx&=\rho_{\alpha,m}^c\frac{1}{\rho_{\alpha,m}^c}\int_e\nabla u\nabla \tilde{h}\,dx=\rho_{\alpha,m}^c\E_{\Jak{m}}(u,\tilde{h})\\
&=\rho_{\alpha,m}^c\kappa\int_e uh\,d\mua=\rho_{\alpha,m}^c\kappa\beta^m\int_e uh\,dx.
\end{align*}
This implies that
\[\int_e\nabla u\nabla h\,dx=\rho_{\alpha,m}^c\kappa\beta^m\int_e uh\,dx\qquad\forall\,h\in H^1(e,dx),\]
hence $\rho_{\alpha,m}^c\kappa\beta^m$ is an eigenvalue of the classical Laplacian $-\Delta$ on $L^2(e,dx)$ subject to Neumann boundary conditions with eigenfunction $u_{\vert_e}$.

Conversely,  if for any $m\geq 1$ and $e\in\Jake{m}\setminus\Jake{m-1}$, $\rho_{\alpha,m}^c\kappa\beta^m$ is an eigenvalue of the classical Laplacian $-\Delta$ on $L^2(e,dx)$ subject to Neumann boundary conditions with eigenfunction $u\in H^1e,dx)$, an analogous computation shows that $\kappa$ is an eigenvalue of $(\E_{\Jak{m}},\D_{\Jak{m}})$ with eigenfunction 
\[\tilde{u}(x):=\left\{\begin{array}{rl}
									u(x),&x\in e,\\
									0,&x\in\Jak{m}\setminus e.
\end{array}\right.\]
Hence, if we denote by $N_e(x)$ the classical Neumann eigenvalue counting function of $-\Delta_{\vert_e}$ we have
\begin{equation}\label{eq: Ja asymp II}
N_{\Jak{m}}(x)=\sum_{k=1}^m \sum_{e\in\Jake{k}\setminus\Jake{k-1}}N_e\left(\rho_{\alpha,k}^c\beta^kx\right)
\end{equation}
and we know form Weyl's theorem (see~\cite{Wey12} for the original version,~\cite{Lap91} for this expression) that 
\[
N_e(x)=\frac{\lambda(e)}{\pi}x^{1/2}+o(x^{1/2})\qquad\text{as }x\to\infty
\]
for each $e\in\Jake{m}$. 
Summing up over all levels we get
\begin{align*}
N_{\Jak{m}}(x)=\frac{\alpha}{\pi}\sum_{k=1}^m3^k\left(\frac{1-\alpha}{2}\right)^{k-1}\left(\rho_{\alpha,k}^c\right)^{1/2}\beta^{k/2}x^{1/2}+o(x^{1/2})
\end{align*}
as $x\to\infty$. 

Since $\beta$ was chosen in~\eqref{eq: property of beta} so that $\sum\limits_{k=1}^\infty\left(3\frac{1-\alpha}{2}\right)^k\beta^{k/2}<\infty$ and $\rho_{\alpha,k}<1$ for all $k\geq 1$,~\eqref{eq: Jam upper} follows with
\[
C_{\alpha,\beta,2}:= \frac{\alpha}{\pi}\sum_{k=1}^\infty3^k\left(\frac{1-\alpha}{2}\right)^{k-1}\beta^{k/2}.
\]
\end{proof}

We recall now the following result from spectral theory of self-adjoint ope\-ra\-tors.
\begin{lemma}\label{lem: cpt resolvent Davies}
Let $(E,D)$ be a Dirichlet form on a Hilbert space $H$ and let $A$ be the non-negative self-adjoint operator on $H$ associated with it. Further, define
\[\kappa(L):=\sup\left\{E(u,u)~\vert~u\in L,~ \norm{u}_H=1\right\},\quad L\subseteq D\text{ subspace},\]
and
\[\kappa_n:=\inf\{\kappa(L)~\vert~ L \text{ subspace of }D,~\dim L=n\}.\]
If the sequence $\{\kappa_n\}_{n=1}^{\infty}$ is unbounded, then the operator $A$ has compact resolvent.
\end{lemma}
\begin{proof}
This follows from~\cite[Theorem 4.5.3]{Davies95} and the converse of~\cite[Theorem 4.5.2]{Davies95}.
\end{proof}

%
%
%
\begin{lemma}\label{lem: Neumann eigenv Kam}
Let $m\geq 0$ and define for any subspace $L\subseteq \D_{\Kak{m}}$
\begin{align*}
\kappa(L)&:=\sup\left\{\E_{\Kak{m}}[u]~\vert~ u\in L,~\int_{\Kak{m}}\abs{u}^2d\mua=1\right\},\\
\kappa_n&:=\inf\{\kappa(L)~\vert~ L \text{ subspace of }\D_{\Kak{m}},~\dim L=n\}.
\end{align*}
Then, there exists a constant $C_U>0$ such that
\begin{equation}\label{eq: Neumann eigenv Kam}
\kappa_{3^m+1}\geq 5^m C_U.
\end{equation}
In particular, the non-negative self-adjoint operator on $L^2(\Kak{m})$ associated with $(\E_{\Kak{m}},\D_{\Kak{m}})$ has compact resolvent.
\end{lemma}
\begin{proof}
The last assertion follows from Lemma~\ref{lem: cpt resolvent Davies} in view of inequality~\eqref{eq: Neumann eigenv Kam}. The proof of this inequality follows the lines of~\cite[Lemma 4.5]{Kaj10} and we include all details for completeness.

Let us consider $L_0:=\{\sum_{w\in\Am}a_w\,{1}_{\Kak{m}}~\vert~ a_w\in\R\}$, which is a $3^m$-dimensional subspace of $\D_{\Kak{m}}$ such that $\E_{\Kak{m}}\vert_{L_0\times L_0}\equiv 0$. Now, take a $(3^m+1)$-dimensional subspace $L\subseteq\D_{\Kak{m}}$ and set $\widetilde{L}:=L_0+L$.
The bilinear form $\E_{\Kak{m}}$ on $\widetilde{L}$ is associated with a non-negative self-adjoint operator $A$ satisfying $\E_{\Kak{m}}(u,v)=\int_{\Kak{m}}(Au)v\,d\mua$ for all $u,v\in\widetilde{L}$.

By the theory of finite-dimensional real symmetric matrices, the $(3^m+1)$-th smallest eigenvalue of $A$ is given by
\[\kappa_A:=\inf\{\kappa(L')~\vert~ L'\text{ subspace of }\widetilde{L},~\dim L'=3^m+1\}.\]
Let $u_A\in\widetilde{L}$ be the eigenfunction corresponding to the eigenvalue $\kappa_A$ and normalize it so that $\int_{\Kak{m}}\abs{u_A}^2d\mua=1$. Note that this function is orthogonal to $L_0$, hence $u_A^+\neq 0\neq u_A^-$. Consider $x,y\in\Ka$ such that $u_A(\Gaw(x)):=\max u_A^+$ and $u_A(\Gaw(y)):=\min u_A^-$. Then, $\abs{u_A(\Gaw(x))-u_A(\Gaw(y))}\geq\abs{u_A(\Gaw(x))}$ and using Corollary~\ref{cor scaling prop EKa} we have that
\begin{align*}
\kappa(L)&\geq \kappa_A=\kappa_A\int_{\Kak{m}}\abs{u_A}^2d\mua=\E_{\Kak{m}}[u_A]\\
&\geq\sum_{w\in\Am}\left(\frac{5}{3}\right)^m\EKa[u_A\circ\Gaw]\\
&\geq\left(\frac{5}{3}\right)^m\sum_{w\in\Am}\frac{\abs{u_A\circ\Gaw(x)-u_A\circ\Gaw(y)}^2}{R(x,y)}\\
&\geq\left(\frac{5}{3}\right)^m\sum_{w\in\Am}\frac{\abs{u_A(\Gaw(x))}^2}{\diam_{R}(\Ka)}\\
&\geq\frac{5^m}{3^m\diam_R(\Ka)}\sum_{w\in\Am}\frac{1}{\mua(\Kak{w})}\int_{\Kak{w}}\abs{u_A(x)}^2d\mua\\
&\geq 5^mC_P\int_{\Kak{w}}\abs{u_A}^2d\mua=5^mC_P,
\end{align*}
with $C_P=(\diam_R(\Ka))^{-1}$ is the inverse of the diameter of $\Ka$ with respect to the resistance metric. Last inequality holds because $3^m\mua(\Kak{w})<1$ for all $w\in\A^m$.
It follows that $\kappa_{3^m+1}\geq 5^mC_P$, as we wanted to prove.
\end{proof}
\begin{prop}\label{prop: upper bound}
There exist $C_{\alpha,2},C_{\alpha,\beta,2}>0$ depending on $\alpha$ and $\beta$, and $x_0 >0$ such that 
\begin{equation}\label{eq: upper bound}
N_N(x)\leq C_{\alpha,2}\,x^{\frac{\ln 3}{\ln 5}}+C_{\alpha,\beta, 2}\,x^{1/2}
\end{equation}
for all $x\geq x_0$.
\end{prop}
\begin{proof}
Let $x_0>C_P$ and $x\geq x_0$. Then we can choose $m\in\N$ such that $C_P5^{m-1}\leq x< C_P5^m$. From Lemma~\ref{lem: Neumann eigenv Kam} we know that 
\begin{equation*}
\kappa_{3^m+1}\geq 5^m C_P> x,
\end{equation*}
hence $N_{\Kak{m}}(x)\leq 3^{m}\leq C_{\alpha,2}\,x^{\frac{\ln 3}{\ln 5}}$, where $C_{\alpha,2}:=3C_P^{-\frac{\ln 3}{\ln 5}}$. Lemma~\ref{lem: Neumann upper +decomp} and Lemma~\ref{lem: Jam upper} lead to~\eqref{eq: upper bound}.
\end{proof}
\subsubsection*{Lower bound}
We recall the definition of the \textit{part of a Dirichlet form}: For any non-empty set $U\subseteq\Ka$, the pair $(\E_U,\D_U)$ given by
\begin{align}\label{defn: part of DF}
\D_U&:=\overline{\mathcal{C}}_U,\quad \mathcal{C}_U:=\{u\in\Da~\vert~ \supp(u)\subseteq U\},\nonumber\\
\E_U&:=\EKa{}_{\vert_{U\times U}},
\end{align}
where the closure is taken with respect to $\norm{\cdot}_{\E_{\Ka,1}}$ is called the part of the Dirichlet form $(\EKa,\Da)$ \textit{on} $U$.

\medskip

Let us write $\Kao:=\Ka\setminus V_0$ and $\Kak{w}^0:=\Ga{w}(\Kao)$ for any $w\in\As$ and set $\Kak{m}^0:=\bigcup\limits_{w\in\Am}\Kak{w}^0$. We consider the Dirichlet forms $(\E_{\Kak{w}^0},\D_{\Kak{w}^0})$ and $(\E_{\Jak{m}},\D_{\Jak{m}})$. 

\begin{lemma}\label{lem: lower +decomp}
For any $m\geq 1$ and $w\in~\Am$, the operators $H_{\Kak{w}^0}$ and $H_{\Kak{m}^0\cup\Jak{m}}$ have compact resolvent and for any $x>0$ we have that
\begin{equation}\label{eq: lower +decomp}
\sum_{w\in\Am}N_{\Kak{w}^0}(x)+N_{\Jak{m}}(x)=N_{\Kak{m}^0\cup\Jak{m}}(x)\leq N_D(x).
\end{equation}
\end{lemma}
\begin{proof}
Note that by definition, $\D_U\subseteq\D_{\Kao}$ and $\E_U=\EKao\vert_{U\times U}$ for both $U\in\{\D_{\Kak{m}^0},\D_{\Kak{m}^0\cup\Jak{m}}\}$ and any $m\in\N$. Since $H_{\Kao}$ has compact resolvent by Theorem~\ref{thm: Spec of Lap}, the minimax principle implies that the operators $H_{\Kak{w}^0}$ and $H_{\Kak{m}^0\cup\Jak{m}}$ also have compact resolvent and the inequality in~\eqref{eq: lower +decomp} holds.

The equality 
\begin{equation}\label{eq: lower 1.decomp}
N_{\Kak{m}^0\cup\Jak{m}}(x)=N_{\Jak{m}}(x)+N_{\Kak{m}^0}(x).
\end{equation}
follows by the same argumentation as in~\cite[Lemma 4.8]{Kaj10}: Let $u\in\D_{\Jak{m}}$. Since $\Ka\setminus\Jak{m}\subseteq\Kak{m}$, $L_{\alpha}:=\supp_{\Ka}(u)\cap\Ka\subseteq\Jak{m}$ and therefore $u\cdot 1_{\Jak{m}}\in C(\Ka)$ and $\supp_{\Ka}(u\cdot 1_{\Jak{m}})\subseteq\Jak{m}$. Since $L_{\alpha}$ is compact and $\Jak{m}$ is open, we know by~\cite[Exercise 1.4.1]{FOT11} that we can find a function $\varphi\in\Da$ such that $\varphi\geq 0$, $\varphi_{\vert_{L_{\alpha}}}\equiv 1$ and $\varphi_{\vert_{\Kak{m}}}\equiv 0$. Then, $u\cdot 1_{\Jak{m}}=u\varphi\in\D_{\Jak{m}}$ and $u\cdot 1_{\Jak{m}}\in\mathcal{C}_{\Jak{m}}$ (recall definition in~\eqref{defn: part of DF}).

Similarly, if $u\in\D_{\Kak{m}^0}$ and $\widetilde{L}_{\alpha}:=\supp_{\Ka}(u)\cap\Ka\subseteq \Kak{m}^0$, we can find $\psi\in\Da$ such that $\psi\geq 0$, $\psi_{\vert_{\widetilde{L}_{\alpha}}}\equiv 1$ and $\psi\vert_{\Jak{m}}\equiv 0$. Thus $u\cdot 1_{\Kak{m}}=u\psi\in\D_{\Kak{m}^0}$ and we have that $\mathcal{C}_{\Kak{m}^0\cup\Jak{m}}=\mathcal{C}_{\Kak{m}}\oplus\mathcal{C}_{\Jak{m}^0}$, both spaces being orthogonal to each other with respect to $\EKa$ and the inner pro\-duct of $L^2(\Ka,\mua)$. Taking the closure with respect to $\E_{\Ka,1}$ we get that $\D_{\Kak{m}^0\cup\Jak{m}}=\D_{\Kak{m}^0}\oplus\D_{\Jak{m}}$, where both spaces keep being orthogonal to each other. Hence~\eqref{eq: lower 1.decomp} follows.

The equality
\begin{equation*}\label{eq: lower 2.decomp}
N_{\Kak{m}^0}(x)=\sum_{w\in\Am}N_{\Kak{w}^0}(x)
\end{equation*}
follows by an analogous argument and the inequality~\eqref{eq: lower +decomp} is therefore proved.
\end{proof}

\begin{lemma}\label{lem: Jam lower}
There exists a constant $C_{\alpha,\beta,1}>0$ depending on $\alpha$ and $\beta$, and $x_0>0$ such that
\begin{equation}\label{eq: Jam upper}
C_{\alpha,\beta,1}x^{1/2}\leq N_{\Jak{m}}(x)
\end{equation}
for all $x\geq x_0$, where $N_{\Jak{m}}(x):=N(x;\E_{\Jak{m}},\D_{\Jak{m}})$.
\end{lemma}
\begin{proof}
Note that in this case $\D_{\Jak{m}}$ can be identified with $\bigoplus\limits_{e\in\J_{\alpha,m}}H^1_0(e,dx)$. The proof is analogous to Lemma~\ref{lem: Jam upper} with $C_{\alpha,\beta,1}=\frac{3\alpha}{\pi}\beta^{1/2}$.
\end{proof}

The proof of the next lemma will make use of the following identification mapping: Let $\{\R^2;S_i,i=1,2,3\}$ be the IFS associated with the Sierpi\'nski gasket $K$ and $\Vs=\bigcup\limits_{n\in\N_0}\bigcup\limits_{w\in\An}S_w(V_0)$. 

\medskip

Recall the IFS $\{\R^2;\Gai,i=1,\ldots,6\}$ associated with $\Ka$ and consider the set $\Was$ defined in~\eqref{eq def Vas}. For any $x\in\Was$ there exists a word $w^x\in\As$ such that $x=\Ga{w^x}(p_i)$ for some $p_i\in V_0$, so we can define
\begin{align*}
\mathcal{I}\colon& \Was\longrightarrow \Vs\\
&\hspace*{0.3cm}x\hspace*{0.3cm}\longmapsto S_{w^x}(p_i).
\end{align*}
This mapping allows us to construct functions in $\Da$ from functions in the domain of the classical Dirichlet form $(\E_K,\D_K)$ on $K$ (see e.g.\cite{KL93} for definitions and details about this form). 

For any function $u\in\D_K$, we define the function $u_{\alpha}\colon \Vas\to\R$ by
\begin{equation}\label{eq: Def ualpha}
u_{\alpha}(x):=\left\{\begin{array}{rl}
									u\circ\mathcal{I}(x),&x\in\Was,\\
									u\circ\mathcal{I}(a_e),&x\in [a_e,b_e],\; e\in\Jae,
\end{array}\right.
\end{equation}
which is well defined since $\mathcal{I}(a_e)=\mathcal{I}(b_e)$ for all $e\in\Jae$. 

For this function it holds that
\[
\EKa[u_{\alpha}]=\lim_{n\to\infty}\left(\frac{3}{5}\right)^n\frac{1}{\rhoan^d}\left(\frac{5}{3}\right)^n\ean^d[u].
\]
Note that $\left(\frac{3}{5}\right)^n\frac{1}{\rhoan^d}$ converges if and only if the series $\sum\limits_{i=1}^\infty\log\left(\frac{5+3\dai}{5}\right)$ converges. By using the Taylor expansion of $\log$ we have that
\[
\abs{\log\left(\frac{5+3\dai}{5}\right)}\leq\frac{3}{5}\alpha\left(\frac{1-\alpha}{2}\right)^{i-1}
\]
hence $\lim\limits_{n\to\infty}\left(\frac{3}{5}\right)^n\frac{1}{\rhoan^d}=:L<\infty$ and
\[
\EKa[u_{\alpha}]=L\cdot\E_K[u]<\infty,
\]
which implies $u_\alpha\in\Da$.
\begin{lemma}\label{lem: 1.ev Kaw}
Let $m\in\N$. There exists $C_D\geq 0$ such that for all $w\in\Am$
\begin{equation}\label{eq: 1.ev Kaw}
\kappa_1(\Kak{w}^0):=\inf_{\substack{u\in\mathcal{C}_{\Kak{w}^0}\\u\neq 0}}\left\{\frac{\EKa[u]}{\norm{u}^2_{L^2(\Kak{m}^0)}}\right\}\leq 5^m C_D.
\end{equation}
\end{lemma}
\begin{proof}
Let $v\in\As$ such that $S_v(K)\subseteq K\setminus V_0$ and consider $u\in\D_K^0$ a function such that $\supp_K (u)\subseteq K\setminus V_0$ and $u\equiv 1$ on $S_v(K)$. Such a function exists by~\cite[Exercise 1.4.1]{FOT11} because $S_v(K)$ is compact and $K\setminus V_0$ is open.

For any $w\in\Am$ we consider the function
\begin{equation*}
u^w(x):=\left\{\begin{array}{rl}
							u_{\alpha}\circ\Ga{w}^{-1}(x),&x\in\Kak{w}^0,\\
							0,&x\in\Ka\setminus\Kak{w}^0,
\end{array}\right.
\end{equation*}
where $u_{\alpha}\in\Da^0$ is defined as in~\eqref{eq: Def ualpha}. Then, $u^w\in\mathcal{C}_{\Kak{w}^0}$ and by Corollary~\ref{cor scaling prop EKa} we have that
\begin{align}\label{align: 1.ev eq 1}
\EKa[u^w]&=\left(\frac{5}{3}\right)^m\sum_{w'\in\Am}\left(\EKa^d[u^w\circ\Ga{w'}]+\left(\frac{2}{1-\alpha}\right)^{2m}\EKa^c[u^w\circ\Ga{w'}]\right)\nonumber\\
&+\left(\frac{1-\alpha}{2}\right)^m\sum_{w\in\Am}\widetilde{\E}_{\Ka,m}^c[u^w\circ\Ga{w}]+\sum_{k=1}^{m-1}\sum_{w\in\A^{k-1}}\left(\frac{2}{1-\alpha}\right)^k\frac{1}{\rho_k^c}E_{\alpha,1}^c[u^w\circ\Ga{w}]\nonumber\\
&=\left(\frac{5}{3}\right)^m\EKa^d[u^w\circ\Ga{w}]
=\left(\frac{5}{3}\right)^m L\E_K[u].
\end{align}
Since $u_{\alpha}\equiv 1$ on $\Kak{v}$ by construction, we also get
\begin{align}\label{eq: 1.ev eq 2}
\int_{\Ka}\abs{u^w(x)}^2d\mua(x) 
\geq\int_{\Kak{v}}d\mua(\Ga{w}(y))
=\mua(\Kak{wv})
\geq \gamma^{\abs{v}}\mua(\Kak{w})
\end{align}
for $\gamma=\beta\frac{1-\alpha}{2}$.

From inequalities~\eqref{align: 1.ev eq 1} and~\eqref{eq: 1.ev eq 2} and the fact that $3^m\mua(\Kak{w})>\frac{1}{2}$, we finally obtain
\begin{align*}
\inf_{\substack{u\in\mathcal{C}_{\Kak{w}^0}\\u\neq 0}}\left\{\frac{\EKa[u]}{\norm{u}^2_{L^2(\Kak{m}^0)}}\right\}&\leq \frac{\EKa[u^w]}{\int_{\Kak{w}^0}\abs{u^w}^2d\mua}\leq \frac{\EKa[u^w]}{\gamma^{\abs{v}}\mua(\Kak{w})}\\
&=\frac{\left(\frac{5}{3}\right)^mL\E_K[u]}{\gamma^{\abs{v}}\mua(\Kak{w})}=\frac{5^mL\E_K[u]}{3^m\gamma^{\abs{v}}\mua(\Kak{w})}\\
&\leq C_D5^m,
\end{align*}
where $C_D:=\frac{2L\E_K[u]}{\gamma^{\abs{v}}}$ is independent of $w$, and inequality~\eqref{eq: 1.ev Kaw} follows.
\end{proof}
Now we can prove the lower bound of the eigenvalue counting function.

\begin{prop}\label{prop: lower bound}
There exist $C_{\alpha,1},C_{\alpha,\beta,1}>0$ depending on $\alpha$ and $\beta$, and $x_0>0$ such that
\[C_{\alpha,1}x^{\frac{\ln 3}{\ln 5}}+ C_{\alpha,\beta,1}x^{1/2}\leq N_D(x)\]
for all $x\geq x_0$.
\end{prop}
\begin{proof}
For $x\geq C_D$, choose $m\in\N_0$ such that $C_D5^{m}\leq x<C_D5^{m+1}$. We know from Lemma~\ref{lem: 1.ev Kaw} that 
\[\kappa_1(\Kak{w}^0)\leq C_D5^m\qquad\forall\,w\in\Am,\]
which implies that $N_{\Kak{w}^0}(x)\geq 1$ for all $w\in\Am$.

By Lemmas~\ref{lem: lower +decomp} and~\ref{lem: Jam upper} we get that
\begin{equation*}
N_D(x)\geq \sum_{w\in\Am}N_{\Kak{w}^0}(x)+N_{\Jak{m}}(x)
\geq C_{\alpha,1}x^{\frac{\ln 3}{\ln 5}}+C_{\alpha,\beta,1}x^{1/2},
\end{equation*}
where $C_{\alpha,1}:=\frac{1}{3}C_D^{\frac{-\ln 3}{\ln 5}}$.
\end{proof}

Finally, we are ready to prove Theorem~\ref{prop estimation evcf} .

\begin{proof}[Proof of Proposition~\ref{prop estimation evcf}]
On one hand, the non-negative self-adjoint operator on $L^2(\Ka,\mua)$ associated to the Dirichlet form $(\EKao,\Dao)$ has compact resolvent by Theorem~\ref{thm: Spec of Lap}.

Further, since $\D_{\Ka^0}\subseteq\Da$ and $\E_{\Kao}$ coincides with $\EKa$ in $\Dao$, it follows from the minimax principle that $N_D(x)\leq N_N(x)$ for any $x\geq 0$.

Finally, let $x_0>\max\{C_P,C_D\}$. Then Propositions~\ref{prop: upper bound} and~\ref{prop: lower bound} provide the first and third inequality in~\eqref{eq prop esti evcf} for all $x\geq x_0$.
\end{proof}

\section{Conclusions and open problems}
An interesting question for further research is the exact expression of the constants involved in Theorem~\ref{prop estimation evcf}. In particular, if the constants $C_{1,\alpha}$ and $C_{2,\alpha}$ of the first term coincide asymptotically, then our result gives directly the second term of the asymptotics of the eigenvalue counting function. We strongly believe that -- with the help of renewal theory -- it is possible to formulate conditions on the parameter $\alpha$ so that $C_{1,\alpha}=C_{2,\alpha}$ holds asymptotically.

\medskip

Another interesting point concerns the diffusion process $(X_t)_{t\geq 0}$ associated to the local regular Dirichlet form $(\EKa,\Da)$. The space-time relation of this process is given by the so--called \textit{walk dimension}. 
If one has Li-Yau type sub-Gaussian estimates for the heat kernel,
\[p(t,x,y)\asymp\frac{C_1}{\mu(B_d(x,t^{1/\delta}))}\exp\left(-C_2\left(\frac{d(x,y)^{\delta}}{t}\right)^{1/(\delta-1)}\right),\]
then the walk dimension coincides with the parameter $\delta$ of the estimate (see~\cite[Example 3.2]{HR03} for the case of the Sierpi\'nski gasket).

\medskip

Spectral dimension and walk dimension are in general related by the so--called \textit{Einstein relation}
\begin{equation}\label{Einstein}
d_Sd_w=2 d_H,
\end{equation}
where $d_H$ denotes the Hausdorff dimension of the set. This relation shows the connection between three fundamental points of view on a set, namely analysis, probability theory and geometry.

\medskip

The Einstein relation has not yet been proven to hold in general but it is known to be truth in the case of the Sierpi\'nski gasket (see e.g.~\cite{Fre12}). The case of Hanoi attractors seems to be quite interesting because of the fact that
\begin{equation*}
d_H(\Ka)<d_S(\Ka)\qquad \forall\,\alpha\in \left(1-\frac{2}{\sqrt{5}},\frac{1}{3}\right).
\end{equation*}
Should $d_w(\Ka)$ exist and the relation in~\eqref{Einstein} hold, one would have $d_w(\Ka)<2$ for $\alpha\in (1-\frac{2}{\sqrt{5}},\frac{1}{3})$. This would mean that the diffusion process associated with the Dirichlet form $(\EKa,\Da)$ for these $\alpha$'s moves faster than two-dimensional Brownian motion. 
Of course, this superdiffusive behavior is brought by the properties of the measure $\mu_{\alpha,\beta}$ giving high conductance to the very small wires in the set. However, the process is still a diffusion and has no jumps. This apparent contradiction with the by now established models for fractal networks arises many interesting questions that should be investigated. Answering these questions may have applications in the design of ``superconductors''.

\medskip

We would also like to note that the resulting process can perhaps be understood as \textit{asymptotically lower dimensional} (ALD). Such processes were first treated in the context of abc-gaskets in~\cite{Hat94}, and studied later on Hambly and Kumagai in~\cite{HK98} for some particular nested fractals.

\section*{Acknowledges}
We thank specially Professors Alexander Teplyaev and Naotaka Kajino for valuable comments and fruitful discussions, as well as for their indispensable advice. We are also thankful for the anonymous feedback pointing out the references~\cite{HK98,HN03,Hat94,MW88,Mur95} and raising the questions of defining ALD processes on Hanoi attractors and of extending this processes to more general objects out of these examples.

\bibliographystyle{amsplain}
\bibliography{AF14_WAHA}
\end{document}